\documentclass[10pt]{article}

\usepackage{graphicx}%
\usepackage{multirow}%
\usepackage{amsmath,amssymb,amsfonts}%
\usepackage{amsthm}%
\usepackage{mathrsfs}%
\usepackage[figuresright]{rotating}%
\usepackage[title]{appendix}%
\usepackage{xcolor}%
\usepackage{textcomp}%
\usepackage{manyfoot}%
\usepackage{booktabs}%
\usepackage{algorithm}%
\usepackage{algorithmicx}%
\usepackage{algpseudocode}%
\usepackage{program}%
\usepackage{listings}%
\usepackage{authblk}

\theoremstyle{thmstyleone}%
\newtheorem{theorem}{Theorem}
\newtheorem{proposition}[theorem]{Proposition}

\newcommand{\ProxB}[1]{{\rm prox}_{h}^{B_k}(#1)}
\newcommand{\argmin}{\operatornamewithlimits{argmin}}

\newtheorem{assumption}{Assumption}{\bf}{\rm}

\theoremstyle{thmstyletwo}%
\newtheorem{remark}{Remark}%

\theoremstyle{thmstylethree}%
\newtheorem{definition}{Definition}%
\newtheorem{lemma}{Lemma}%

\title{Inexact proximal DC Newton-type method for nonconvex composite functions}

\author[1]{Shummin Nakayama\footnote{E-mail:~\texttt{snakayama@uec.ac.jp}}}
\author[2]{Yasushi Narushima \footnote{E-mail:~\texttt{narushima@ae.keio.ac.jp}}}
\author[3]{Hiroshi Yabe \footnote{E-mail:~\texttt{yabe@rs.tus.ac.jp}}}
\affil[1]{Info-Powered Energy System Research Centern, The University of Electro-Communications
\footnote{1-5-1 Chofugaoka, Chofu, Tokyo 182-8585, Japan}}%
\affil[2]{Department of Industrial and Systems Engineering, Keio University
\footnote{3-14-1 Hiyoshi, Kouhoku-ku, Yokohama 223-8522, Japan}}
\affil[3]{Center for Data Science, Tokyo University of Science
\footnote{1-3 Kagurazaka, Shinjuku-ku, Tokyo 162-8601, Japan}}

\begin{document}

\maketitle

\begin{abstract}
We consider a class of difference-of-convex (DC) optimization problems where the objective function is the sum of a smooth function and a 
possibly nonsmooth DC function. 
The application of proximal DC algorithms to address this problem class is well-known. In this paper, we combine a proximal DC algorithm with an inexact proximal Newton-type method to propose an inexact proximal DC Newton-type method.
We demonstrate global convergence properties of the proposed method. 
In addition, we give a memoryless quasi-Newton matrix for scaled proximal mappings and  consider a two-dimensional system of semi-smooth equations that arise in calculating scaled proximal mappings. 
To efficiently obtain the scaled proximal mappings, we adopt a semi-smooth Newton method to inexactly solve the system. 
Finally, we present some numerical experiments to investigate the efficiency of the proposed method, which show that the proposed method outperforms existing methods.

\end{abstract}

\section{Introduction}
In this paper, we consider minimization of the following composite function:
\begin{equation}\label{min}
\min_{x\in \mathbb{R}^n}\quad f(x) := g(x)+h(x),
\end{equation}
where $g:\mathbb{R}^n\to\mathbb{R}$ is an $L$-smooth function, and 
$h:\mathbb{R}^n\to\mathbb{R}\cup\{\infty\}$ is a difference-of-convex (DC) function:
\begin{equation}\label{DCregularizer}
h(x) = h_1(x) - h_2(x),
\end{equation}
where $h_1:\mathbb{R}^n\to\mathbb{R}\cup\{\infty\}$ is a proper lower semi-continuous (lsc) convex function and $h_2:\mathbb{R}^n\to{\mathbb{R}}$ is a continuous convex function. 
This problem appears in statistics and machine learning. Typically in machine learning, $g$ is a loss function, such as the least square function, the logistic loss function, or the nonconvex quadratic function, and $h$ is a regularizer. 
Although the well-known $\ell_1$ regularizer is intended as an approximation of the $\ell_0$-norm, it is convex and not sufficient as the approximation. 
Thus, several improved approximations have been proposed, including the Smoothly Clipped Absolute Deviation (SCAD)~\cite{JASA_Fan2001,gong2013general}, the Minimax Concave Penalty (MCP)~\cite{gong2013general,AS_Zhang2010}, the $\ell_{1-2}$ regularizer~\cite{yin2015minimization}, the truncated $\ell_1$ regularizer~\cite{MP_Gotoh2018,lu2018sparse,nakayama2021superiority}, the Capped $\ell_1$ regularizer~\cite{gong2013general,zhang2010analysis}, and the Log-Sum Penalty~\cite{candes2008enhancing,gong2013general}.  
Note that these regularizers are DC functions formulated by \eqref{DCregularizer}. We note that a usual DC programming requires $g$ to be DC functions, i.e.,
\begin{equation}\label{DC_g}
g(x) = g_1(x) - g_2(x)
\end{equation}
where $g_1$ and $g_2$ are convex. Then, \eqref{min} can be regarded as DC optimization of the form
\[
f(x) = (g_1(x)+h_1(x)) - (g_2(x)+h_2(x)).\]
If necessary,  we can consider DC decomposition of $g$, but this paper directly deal with the nonconvex form.

In the case $h(x)=h_1(x)$, the proximal gradient method can be used~\cite{beck2017first,fukushima1981generalized}. 
The Fast Iterative Shrinkage-Thresholding Algorithm (FISTA)~\cite{Beck2009}, which is a proximal gradient method with Nesterov's acceleration scheme, was proposed to accelerate the approach. As an alternative acceleration scheme, a proximal Newton-type method has also been studied~\cite{SIOPT_Lee2014}. Although usual proximal mappings can be easily obtained for some special cases~\cite{beck2017first,combettes2011proximal}, the computing cost of scaled proximal mapping is very expensive. Hence, inexact proximal Newton-type methods, which inexactly calculate scaled proximal mappings, have been proposed~(see, for example, \cite{byrd2016inexact,lee2019inexact,li2017inexact,COAP_Nakayama2021,scheinberg2016practical}). 
The proximal Newton-type method is also known as the Successive Quadratic Approximation (SQA) method. 

The DC Algorithm (DCA) is a classical algorithm~\cite{Acta_Tao1977} for solving DC optimization problems. In our problem settings, the proximal DCA (pDCA)~\cite{MP_Gotoh2018} can be  used. To accelerate the algorithm, Wen et al.~\cite{COAP_Wen2018} incorporated Nesterov's acceleration scheme into the pDCA, creating pDCA with extrapolation (pDCAe). Another acceleration approach is pDCA based on the Newton method \cite{rakotomamonjy2015dc}. Recently,  Liu and Takeda~\cite{OptOnline_Liu2021} extended an inexact SQA method to DC optimization.

In this paper, we propose an inexact proximal DC Newton-type method. The findings and contributions of this paper are summarized as follows.
\begin{itemize}
\item 
We propose an inexact proximal DC Newton-type method, which is an extension of the inexact proximal Newton-type method~\cite{COAP_Nakayama2021} to DC optimization problems, and show its global convergence. 
Specifically, our key contributions are concrete choices for quasi-Newton matrices of the scaled proximal mappings and an efficient numerical method for solving the subproblem. 
In the inexact SQA method proposed by Liu and Takeda~\cite{OptOnline_Liu2021}, the method requires {strong} convexity of $g$ to solve the subproblem and obtain the scaled proximal mappings. 
On the other hand, our method (Algorithm~\ref{alg:proximal-DC-Newton}) can {be directly applied to \eqref{min} without assuming strong} convexity of $g$.
\item 
As mentioned above, the computing cost of scaled proximal mappings is expensive in general cases. 
In this paper, we consider (a) concrete choices for quasi-Newton matrices in scaled proximal mappings 
and (b) an efficient numerical method for computing scaled proximal mappings. Specifically, we deal with a modification of memoryless quasi-Newton matrices proposed by Nakayama et.~al.~\cite{COAP_Nakayama2021}. Then, the scaled proximal mappings can be obtained by solving a two-dimensional system of semi-smooth equations~\eqref{def:L}. To solve the semi-smooth equation, we use the semi-smooth Newton method (Algorithm~\ref{alg:semismoothNewton}). 
\item
In numerical experiments, we compare the proposed method with other existing methods and show the efficiency of the proposed method.
\end{itemize}

This paper is organized as follows. 
In Section~\ref{sec:propose}, we propose an inexact proximal DC Newton-type method.
We briefly introduce the inexact proximal Newton-type method~\cite{COAP_Nakayama2021} in Section~\ref{sec:InexactNewton}, and  we extend the inexact proximal Newton-type method to DC optimization problems in Section~\ref{sec:proposedmethod}. In Section~\ref{sec:conv}, we show the global convergence properties of the proposed method. In Section~\ref{sec:B}, we introduce the memoryless quasi-Newton formula~\cite{COAP_Nakayama2021} (Section~\ref{subsec:B}) and give an efficient method for computing scaled proximal mappings (Section~\ref{subsec:prox}). In Section~\ref{sec:numerical}, we present some numerical experiments to show the efficiency of the proposed method in comparison with existing methods. Finally, we conclude and provide remarks in Section~\ref{sec:con}.

Throughout this paper, we denote the identity matrix, the $\ell_2$ norm, and the $\ell_1$ norm by $I\in\mathbb{R}^{n\times n}$, $\|\cdot\|$, and $\|\cdot\|_1$, respectively.
For a symmetric positive definite matrix $A$ and a proper convex function $\tilde{h}$, a scaled proximal mapping is defined by
\begin{equation*}\label{ProxB}
{\rm Prox}_{\tilde{h}}^{A}(y) \equiv \argmin_{x\in \mathbb{R}^n} \left(\tilde{h}(x) + \frac{1}{2}\|x-y\|_{A}^2\right),
\end{equation*}
where $\|x\|_A=\sqrt{x^TAx}$. 
In the case $A=I$, we omit the {superscript} and it is 
{the} usual proximal mapping. 
Finally, $\partial \tilde{h}(\cdot)$ is the subdifferential of a convex function $\tilde{h}$, 
$\partial^C\mathcal{L}(\cdot)$ is the Clarke 
{differential} of a nonlinear {mapping} $\mathcal{L}$, 
{
and we denote the $i$-th component of a vector $v$ by $(v)_i$.
}

\section{Inexact proximal DC Newton-type method}\label{sec:propose}
We first introduce the inexact proximal Newton-type method~\cite{COAP_Nakayama2021} in Section~\ref{sec:InexactNewton}. Then in Section~\ref{sec:proposedmethod}, we extend the method to DC optimization problems and propose an inexact proximal DC Newton-type method. 

\subsection{Inexact proximal Newton-type method}\label{sec:InexactNewton}
Consider the special case $h(x)=h_1(x)$, namely
\begin{equation*}\label{min_convex}
\min_{x\in \mathbb{R}^n}\quad f(x) := g(x)+h_1(x).
\end{equation*}
For solving the problem, we briefly {review} a framework of the inexact proximal Newton-type method proposed by Nakayama et al.~\cite{COAP_Nakayama2021}. 
The method generates a sequence $\{x_k\}$ according to
\begin{equation*}
x_{k+1}=x_k+\eta_kd_k,
\end{equation*}
where $x_k\in\mathbb{R}^n$ is the $k$-th approximation to a solution, $\eta_k>0$ is a step size and $d_k\in\mathbb{R}^n$ is a search direction given by
\begin{equation}\label{search}
d_k=x_k^+ - x_k.
\end{equation}
Here, $x_k^+$ is an approximation solution of the following subproblem
\begin{equation}\label{sub_conv}
\argmin_{x\in \mathbb{R}^n}~
g(x_k)+\nabla g(x_k)^T(x-x_k)+\frac{1}{2}(x-x_k)^TB_k(x-x_k) + h_1(x),
\end{equation}
which is the sum of $h_1$ and a quadratic model of $g$ at $x_k$,
where $B_k\in\mathbb{R}^{n\times n}$ is symmetric positive definite and an approximation of the Hessian $\nabla^2g(x_k)$.
If the above minimization problem is solved exactly, then 
\begin{equation}\label{prox_exact1}
x_k^+ = {\rm Prox}_{h_1}^{B_k}(x_k - H_k\nabla g(x_k))
\end{equation}
holds, where $H_k=B_k^{-1}$. 
The optimality condition of \eqref{sub_conv} is given by 
\begin{equation*}\label{opt_prox_conv}
0 \in  \nabla g(x_k)+B_k(x_k^+ -x_k) + \partial h_1(x_k^+).
\end{equation*}
If we solve \eqref{sub_conv} inexactly, 
then there exists a gradient residual $r_k$ such that
\begin{equation*}\label{sub_opt_conv}
r_k \in  \nabla g(x_k) +B_k(x_k^+ -x_k) + \partial h_1(x_k^+).
\end{equation*}
We accept $x_k^+$ as an approximation solution of \eqref{sub_conv} if 
\begin{equation}\label{inexact}
\|r_k\|_{H_k} \leq (1-\theta_k)\|x_k^+-x_k\|_{B_k},\quad \theta_k\in[\bar\theta,1],
\end{equation}
is satisfied, where $\theta_k$ is a parameter and $\bar\theta\in(0,1]$ is a constant. 
We can find a simple example  to  achieve the above inexact condition in \cite[Section~2]{arXive_Liu2017} and \cite[Section~4]{COAP_Nakayama2021}. 

\subsection{
Inexact proximal DC Newton-type method
}
\label{sec:proposedmethod}
In this section, we propose a new algorithm, which is an extension of the inexact proximal Newton-type method introduced in Section \ref{sec:InexactNewton}. We first consider the following linear approximation of $h_2$ at $x_k$:
\[h_2(x) \approx h_2(x_k) + \xi_k^T(x - x_k),\]
where $\xi_k\in\partial h_2(x_k)$ is a subgradient. 
Combining the above and \eqref{sub_conv}, we have the following subproblems where $x_k^+$ is the solution:
\begin{align}
\argmin_{x\in \mathbb{R}^n}~
(\nabla g(x_k)-\xi_k)^T(x-x_k) + \frac{1}{2}\|x - x_k\|_{B_k}^2
+h_1(x).
\label{w_prox_grad} 
\end{align}
Similarly to \eqref{prox_exact1},  if \eqref{w_prox_grad} is solved exactly, then 
\begin{equation}\label{ProxB_exact}
x_k^+ = \ProxB{x_k-H_k(\nabla g(x_k)-\xi_k)}
\end{equation}
holds. 
If we solve \eqref{w_prox_grad} inexactly, namely, 
\begin{equation}\label{inexact_prox}
x_k^+ \approx \ProxB{x_k-H_k(\nabla g(x_k)-\xi_k)},
\end{equation}
then there exists a gradient residual $r_k$ such that
\begin{equation}\label{sub_opt}
r_k \in  \nabla g(x_k)-\xi_k+B_k(x_k^+ -x_k) + \partial h_1(x_k^+).
\end{equation}
We accept $x_k^+$ as (\ref{inexact_prox}) if \eqref{inexact} is satisfied. 
We give a concrete choice of $r_k$ and a numerical method in Section~\ref{subsec:prox}.

We define a search direction $d_k$ by \eqref{search}.
For the line search, we select the step size $\eta_k$ satisfying the condition 
\begin{equation}\label{linecond}
	f(x_k+\eta_k d_k) \leq f(x_k)+\delta\eta_k((\nabla g(x_k)-\xi_k)^Td_k + h_1(x_k^+) - h_1(x_k))
\end{equation}
by using backtracking scheme, where $\delta\in(0,1)$. 
Summarizing the above arguments, we give Algorithm~\ref{alg:proximal-DC-Newton}.
\begin{algorithm}[h]
\caption{Inexact proximal DC Newton-type method}
\label{alg:proximal-DC-Newton}
\begin{algorithmic}
\Require{$x_0\in {\rm dom}(f)$, $\delta\in(0,1)$, $\bar\theta\in(0,1]$, $0<\beta_{min}\leq\beta_{max}<1$, $\varepsilon>0$}
\For{$k=0,1,2...$}
\State Choose $B_k$. 
\State Choose $\xi_k\in\partial h_2(x_k)$ and $\theta_k\in[\bar\theta, 1]$.
\State Compute $x_k^+$ 
  satisfying \eqref{sub_opt} and \eqref{inexact}.
\State $d_k \leftarrow x_k^+ - x_k$
\If{
the stopping condition $\|d_k\|\leq \varepsilon$ is satisfied 
}
\State stop.
\EndIf
\State $\rho\leftarrow 1$
\While{condition \begin{align*}
f(x_k+\rho d_k) \leq f(x_k)+\delta\rho((\nabla g(x_k)-\xi_k)^Td_k + h_1(x_k^+) - h_1(x_k))
\end{align*} 
\State is not satisfied}
\State Choose $\beta\in[\beta_{min},\beta_{max}]$.
\State $\rho \leftarrow \beta\rho$
\EndWhile
\State $\eta_k \leftarrow \rho$
\State $x_{k+1} \leftarrow x_k + \eta_kd_k$
\EndFor
\end{algorithmic}
\end{algorithm}

In Algorithm~\ref{alg:proximal-DC-Newton}, we adopt $\|d_k\|\leq \varepsilon$ as a stopping condition, 
because $d_k=0$ implies that $x_k$ is a critical point (see Theorem~\ref{lem_opt} in Section~\ref{sec:conv}). 
Though we can use another stopping condition, we must in this case use the same stopping condition in the algorithm for finding $x_k^+$ (namely, Algorithm~\ref{alg:semismoothNewton} in Section~\ref{subsec:prox}). 

We note that Algorithm~\ref{alg:proximal-DC-Newton} is identical to the inexact proximal Newton-type method~\cite{COAP_Nakayama2021} when $h_2=0$ ($\xi_k=0$), and that it corresponds to pDCA~\cite{MP_Gotoh2018} when $B_k=L\times I$ and $\eta_k=1$ for all $k$. 
Although the algorithm is similar to the method proposed by Liu and Takeda~\cite{OptOnline_Liu2021}, the inexact role of the subproblem and the line search scheme are different.

\section{Convergence properties}\label{sec:conv}
In this section, we show the global convergence of Algorithm \ref{alg:proximal-DC-Newton}. 
Throughout this paper, we use the following definition~\cite{MP_Gotoh2018,Acta_Tao1977}.
\begin{definition}
If 
\begin{equation}\label{crit}
0 \in \nabla g(x^\ast) + \partial h_1(x^\ast) - \partial h_2(x^\ast)
\end{equation}
holds, then we call $x^\ast$ a critical point of \eqref{min}.
\end{definition}

Note that the above condition is a weaker condition than the directional stationary condition, which implies that $\tilde{x}$ satisfies
\begin{equation}\label{d-stat}
{f'(\tilde{x};d):=
}
\lim_{\eta\to {0_+}} \frac{f(\tilde{x}+\eta d) - f(\tilde{x})}{\eta} \geq 0
\quad \text{for all}\quad 
d\in\mathbb{R}^n.
\end{equation}
We call $\tilde{x}$ a directional stationary point of \eqref{min} if \eqref{d-stat} holds. 
Almost all pDCA type methods aim to find a critical point. We show that the proposed method converges to a critical point. 

To show the global convergence, we make the following standard assumptions. 
\begin{assumption}\label{ass:Lip}
\begin{enumerate}
\item
The function $g:\mathbb{R}^n\to{\mathbb{R}}$ is continuously differentiable and its gradient $\nabla g$ is Lipschitz continuous, namely, there exists a positive constant $L$ such that
\begin{equation}\label{Lip_Ass}
\|\nabla g(u)-\nabla g(v)\| \leq L\|u-v\|,\quad \forall u,v\in \mathbb{R}^n.
\end{equation}
\item
$h_1:\mathbb{R}^n\to\mathbb{R}\cup\{\infty\}$ is a proper lsc convex function and $h_2:\mathbb{R}^n\to{\mathbb{R}}$ is a continuous convex function.
\end{enumerate}
\end{assumption}
\begin{assumption}\label{ass:mat}
There exist positive constants $m$ and $M$ such that
\begin{equation}\label{uniformly_mat}
m \|u\|^2 \leq \|u\|_{B_k}^2= u^TB_ku \leq M \|u\|^2\quad \forall  u\in \mathbb{R}^n.
\end{equation}
\end{assumption}

We first give the following lemma. This is a DCA version of \cite[Lemma~2]{COAP_Nakayama2021}, so the proof given in Appendix \ref{Appendix:lemma1} for self-containedness is similar to the proof in \cite{COAP_Nakayama2021}. 
\begin{lemma}\label{lem_line}
Suppose that Assumptions \ref{ass:Lip}--\ref{ass:mat} hold. Let the sequence $\{x_k\}$ be generated by Algorithm \ref{alg:proximal-DC-Newton}. Then for all $\eta\in(0,1]$, 
\begin{equation}\label{lem_alpha}
f(x_k+\eta d_k) \leq f(x_k)+\eta((\nabla g(x_k)-\xi_k)^Td_k + h_1(x_k^+) - h_1(x_k))+\frac{\eta^2L}{2}\|d_k\|^2,
\end{equation}
and
\begin{equation}\label{decrease2}
(\nabla g(x_k)-\xi_k)^Td_k + h_1(x_k^+) - h_1(x_k) \leq - \bar\theta \|d_k\|_{B_k}^2
\end{equation}
are satisfied.
\end{lemma}
\begin{remark}
Since it follows from \eqref{uniformly_mat}, \eqref{lem_alpha} and \eqref{decrease2} that
\begin{equation}\label{descent}
f(x_k+\eta d_k) \leq f(x_k) +
\eta\left(\frac{\eta L}{2m} - \bar\theta\right)\|d_k\|^2_{B_k}
\end{equation}
holds, 
$f(x_k+\eta d_k) \leq f(x_k)$ if $\eta < \frac{2m}{L}\bar\theta$. Therefore, 
the sequence $\{f(x_k)\}$ is decreasing.
\end{remark}
The next lemma implies that $\eta_k$ is bounded away from 0.  
This lemma is a DCA version of \cite[Lemma~3]{COAP_Nakayama2021} and follows from a similar proof (see Appendix~\ref{Appendix:lemma2}). 
\begin{lemma}\label{alpha_lemma}
Suppose that all assumptions of Lemma \ref{lem_line} are  satisfied.
Then there exists $\eta_k$ such that the line search condition \eqref{linecond} is satisfied.
Moreover, the following holds:
\begin{equation}\label{alpha_lem_back}
\bar\eta \equiv \beta_{min}\min\left\{1, \frac{2m}{L}\bar\theta(1-\delta)\right\}\leq\eta_k \leq 1.
\end{equation}
\end{lemma}
Using 
Lemmas \ref{lem_line}--\ref{alpha_lemma}, we obtain the following global convergence theorems of Algorithm~\ref{alg:proximal-DC-Newton}.
\begin{theorem}\label{lem_opt}
Suppose that Assumptions \ref{ass:Lip}--\ref{ass:mat} hold. Let the sequence $\{x_k\}$ be generated by Algorithm \ref{alg:proximal-DC-Newton}. Then, the following statements hold:
\begin{enumerate}
\item[(i)]
If $d_k=0$, then $x_k$ is a critical point of $\eqref{min}$. 
\item[(ii)]
If {$f$ is directionally differentiable and} 
 $x_k$ is a directional stationary point of $\eqref{min}$, 
then $d_k=0$.
\end{enumerate}
\end{theorem}
\begin{proof}
(i) If $d_k=0$, then it follows from \eqref{inexact} that $r_k=0$, and, hence, $x_k^+=x_k$ holds.
Thus, the condition \eqref{sub_opt} and $\xi_k\in \partial h_2(x_k)$ yield
\[0 \in \nabla g(x_k) + \partial h_1(x_k) - \xi_k \subseteq \nabla g(x_k) + \partial h_1(x_k) - \partial h_2(x_k).\]
\noindent
(ii) 
It follows from \eqref{descent} that
\[
\frac{f(x_k+\eta d_k) - f(x_k)}{\eta} 
 \leq -\bar\theta \|d_k\|^2_{B_k} +\frac{\eta L}{2m}\|d_k\|_{B_k}^2.
\]
Since \eqref{d-stat} with $\tilde{x}=x_k$ yields
\[
0\leq \lim_{\eta\to{0_+}}\frac{f(x_k+\eta d_k) - f(x_k)}{\eta}
\leq -\bar\theta \|d_k\|^2_{B_k},
\]
we have $d_k=0$ by \eqref{uniformly_mat}. 

Therefore, the proof is complete.
\end{proof}

The theorem suggests that $d_k\neq0$ is possible when $x_k$ is a critical point, but not when it is a directional stationary point. This is a desirable property, because the directional stationary condition is a stronger condition than \eqref{crit}.

In the rest of this section, we assume $\|d_k\|\ne 0$ for all $k$. Otherwise, a critical point has already been found. 
The next theorem means that the proposed method converges globally to a critical point.

\begin{theorem}\label{mainThm1}
Suppose that Assumptions \ref{ass:Lip}--\ref{ass:mat} hold. Let the sequence $\{x_k\}$ be generated by Algorithm \ref{alg:proximal-DC-Newton}. 
If the objective function $f$ is bounded below, then
\begin{equation}\label{dto0}
\lim_{k\to\infty}\|d_k\|=0.
\end{equation}
Furthermore, if $\{x_k\}$ is bounded, then any accumulation point of  $\{x_k\}$ is  a critical point of $\eqref{min}$.
\end{theorem}
\begin{proof}
From Lemma \ref{alpha_lemma}, there exists a step size satisfying the line search  condition \eqref{linecond}.
Therefore, by \eqref{linecond}, \eqref{uniformly_mat}, \eqref{decrease2}, and \eqref{alpha_lem_back}, we have
\begin{align*}
f(x_{k+1})-f(x_k) & \leq \delta \eta_k ((\nabla g(x_k)-\xi_k)^Td_k + h_1(x_k^+) - h_1(x_k))\\
  &\leq -\delta\bar\eta \bar\theta \|d_k\|_{B_k}^2\\
  &\leq -\delta\bar\eta \bar\theta m\|d_k\|^2\\
  &\leq0.
\end{align*}
Hence the sequence $\{f(x_k)\}$ is nonincreasing. 
Since $f$ is bounded below, the sequence $\{f(x_k)\}$ must converge to some limit, 
which implies that
\[\lim_{k\to\infty} \left\{f(x_{k+1})-f(x_k)\right\}=0.\]
Thus,  \eqref{dto0} holds.
It follows from \eqref{inexact}, \eqref{uniformly_mat} and \eqref{dto0} that
\[\lim_{k\to\infty}\|r_k\|=0.\]
Let $\bar{x}$ be an accumulation point of $\{x_k\}$. 
Since $\partial h_1$ is closed and $\xi_k\to\bar\xi\in\partial h_2(\bar{x})$,  it follows from \eqref{search}, \eqref{sub_opt}, and \eqref{dto0} that
\[0\in \nabla g(\bar{x})+\partial h_1(\bar{x}) - \bar\xi \subseteq \nabla g(\bar{x})+\partial h_1(\bar{x}) - \partial h_2(\bar{x}),\]
completing the proof.
\end{proof}

\section{Choices of $B_k$ and computing scaled proximal mappings}\label{sec:B}
We present concrete choices of $B_k$ in Section~\ref{subsec:B} and a numerical method for solving subproblem \eqref{w_prox_grad} in Section~\ref{subsec:prox}.

\subsection{Memoryless quasi-Newton matrices}\label{subsec:B}
In this subsection, we establish concrete choices of $B_k$.  
For this purpose, we first consider the quasi-Newton updating formula and introduce the modified spectral scaling Broyden family proposed by Nakayama et al.~\cite[Eq.~(13)]{COAP_Nakayama2021}:
\begin{equation}\label{SBroyden_B1}
B_k  = B_{k-1} - \frac{ B_{k-1}s_{k-1}s_{k-1}^TB_{k-1} }{ s_{k-1}^TB_{k-1} s_{k-1} } + \gamma_k\frac{z_{k-1}z_{k-1}^T}{ s_{k-1}^Tz_{k-1} }
  + \hat\phi_k \hat{v}_{k-1} \hat{v}_{k-1}^T,
\end{equation}
\begin{equation*}\label{SBroyden_v1}
\hat{v}_{k-1} = \sqrt{s_{k-1}^T B_{k-1}s_{k-1}}\left(\dfrac{z_{k-1}}{s_{k-1}^Tz_{k-1}}-\dfrac{ B_{k-1}s_{k-1}}{s_{k-1}^TB_{k-1} s_{k-1}}\right),
\end{equation*}
where $\gamma_k>0$ is a scaling parameter, $\hat\phi_k$ is a parameter of the Broyden family,  
\begin{align}
s_{k-1}=x_k-x_{k-1}\quad\text{and}\quad z_{k-1}=(\nabla g(x_k) - \nabla g(x_{k-1}))+ \nu_ks_{k-1},
\label{def:sz}
\end{align}
where $\nu_k\geq0$ is a modified parameter such that $\nu_k\le\bar\nu$ and
\begin{equation}\label{sz>m}
s_{k-1}^Tz_{k-1}=s_{k-1}^T((\nabla g(x_k) - \nabla g(x_{k-1}))+\nu_k s_{k-1}) \ge  \underline{\nu} \|s_{k-1}\|^2
\end{equation}
hold for fixed constants $\underline\nu$ and $\bar\nu$. 
{
In our numerical experiments (Section~\ref{sec:numerical}), to  achieve \eqref{sz>m}, we use 
\begin{equation}\label{nu}
\nu_k=\begin{cases}
0, & \text{if}~s_{k-1}^Ty_{k-1} \geq  \tilde\nu\|s_{k-1}\|^2\\
\max\left\{0,-\frac{s_{k-1}^Ty_{k-1}}{s_{k-1}^Ts_{k-1}}\right\} +  \tilde\nu, & \text{otherwise},
\end{cases}
\end{equation}
which is called Li-Fukushima's regularization~\cite{LiFu2001}, where $\tilde\nu>0$ is a constant parameter.
We note that  \eqref{sz>m} is satisfied with $\underline\nu=\tilde\nu$ and $\bar\nu=L+\tilde\nu$ when \eqref{Lip_Ass} holds.
}
If we choose $\hat\phi_k$ such that $\hat\phi_k>\hat\phi_k^\ast$, then $B_k$ updated by \eqref{SBroyden_B1}  is symmetric positive definite, where
\begin{equation}\label{Bphi*}
\hat\phi_k^\ast =-\frac{(s_{k-1}^Tz_{k-1})^2}{(s_{k-1}^TB_{k-1}s_{k-1})(z_{k-1}^TB_{k-1}^{-1}z_{k-1}) - (s_{k-1}^Tz_{k-1})^2 }<0.
\end{equation}
Furthermore, Nakayama et al.~\cite{COAP_Nakayama2021} proposed the memoryless modified spectral scaling Broyden family, which is given by \eqref{SBroyden_B1} with $B_{k-1}=I$. In this paper, we improve the method by applying a sizing technique, and proposing \eqref{SBroyden_B1} with $B_{k-1}=\tau_kI$:
\begin{equation}\label{SBroyden_B}
B_k  = \tau_kI - \tau_k\frac{s_{k-1}s_{k-1}^T}{ s_{k-1}^Ts_{k-1} } + \gamma_k\frac{z_{k-1}z_{k-1}^T}{ s_{k-1}^Tz_{k-1} }
  + \tau_k\phi_k v_{k-1} v_{k-1}^T,
\end{equation}
\begin{equation*}\label{SBroyden_v}
v_{k-1} = \sqrt{s_{k-1}^Ts_{k-1}}\left(\dfrac{z_{k-1}}{s_{k-1}^Tz_{k-1}}-\dfrac{ s_{k-1}}{s_{k-1}^T s_{k-1}}\right),
\end{equation*}
where $\tau_k>0$ is a sizing parameter. 
Note that sizing is a standard  technique for the quasi-Newton method (see, for example \cite{nocedal2006numerical,sun2006optimization}). 
If we choose $\phi_k$ such that $\phi_k>\phi_k^\ast$, then $B_k$ updated by \eqref{SBroyden_B}  is symmetric positive definite, where
\begin{equation*}\label{phi*}
\phi_k^\ast =-\frac{(s_{k-1}^Tz_{k-1})^2}{(s_{k-1}^Ts_{k-1})(z_{k-1}^Tz_{k-1}) - (s_{k-1}^Tz_{k-1})^2 }<0,
\end{equation*}
which is \eqref{Bphi*} with $B_{k-1}=\tau_kI$.
The inverse of (\ref{SBroyden_B})  is given by
\begin{equation*}\label{SBroyden_H}
H_k = \frac{1}{\tau_k}I - \frac{1}{\tau_k}\frac{z_{k-1}z_{k-1}^T}{z_{k-1}^Tz_{k-1}} + \frac{1}{\gamma_k} \frac{s_{k-1} s_{k-1}^T}{s_{k-1}^Tz_{k-1}} + \frac{1}{\tau_k}\phi_k^Hw_{k-1}w_{k-1}^T,
\end{equation*}
\begin{equation*}\label{Sw}
w_{k-1} = \sqrt{z_{k-1}^Tz_{k-1}}\left(\dfrac{s_{k-1}}{s_{k-1}^Tz_{k-1}}-\dfrac{z_{k-1}}{z_{k-1}^Tz_{k-1}}\right),
\end{equation*}
where 
\[\phi_k^H = \frac{\phi_k^\ast(1-\phi_k)}{\phi_k^\ast - \phi_k}.\]
To obtain the uniformly positive definiteness of $B_k$, we restrict the interval of $\phi_k$ to
\begin{equation}\label{spd_phi}
\overline\phi_1 \phi_k^\ast \leq \phi_k \leq  \overline\phi_2,
\end{equation}
where  $0\leq \overline\phi_1<1$ and $\overline\phi_2>0$ are constants.
We choose $\gamma_k$ and $\tau_k$ satisfying the conditions
\begin{equation}\label{<gamma<}
\underline\gamma\leq \gamma_k \leq \overline\gamma
\quad\text{and}\quad
\underline\tau\leq \tau_k \leq \overline\tau,
\end{equation}
where  $\underline\gamma$, $\overline\gamma$, $\underline\tau$, and $\overline\tau$  are positive constants such that $0<\underline\gamma\leq \overline\gamma$ and $0<\underline\tau\leq \overline\tau$ hold.
Then the following proposition holds. 
\begin{proposition}\label{Prop_mat}
Suppose Assumption \ref{ass:Lip} is satisfied, and $B_k$ is given by $\eqref{SBroyden_B}$.
If $(\ref{sz>m})$, $\eqref{spd_phi}$ and $(\ref{<gamma<})$ hold,
then \eqref{uniformly_mat} holds.
\end{proposition}
We note that Proposition~\ref{Prop_mat} with $\tau_k=1$ is proven in \cite[Proposition 1]{COAP_Nakayama2021}. Dividing \eqref{SBroyden_B} by $\tau_k$, we have
\[
\frac{1}{\tau_k}
B_k  = I - \frac{s_{k-1}s_{k-1}^T}{ s_{k-1}^Ts_{k-1} } +\frac{\gamma_k}{\tau_k}\frac{z_{k-1}z_{k-1}^T}{ s_{k-1}^Tz_{k-1} }
  + \phi_k v_{k-1} v_{k-1}^T.
\]
Then, since $\frac{\underline\gamma}{\overline\tau}\leq \frac{\gamma_k}{\tau_k} \leq \frac{\overline\gamma}{\underline\tau}$ holds form \eqref{<gamma<}, 
we can prove the proposition in almost the same way as \cite[Proposition 1]{COAP_Nakayama2021}.

From Theorem \ref{mainThm1} and Proposition \ref{Prop_mat}, we have the following convergence result.
\begin{theorem}
Suppose 
Assumption \ref{ass:Lip} holds. 
Let the sequence $\{x_k\}$ be generated by Algorithm \ref{alg:proximal-DC-Newton} with \eqref{SBroyden_B}. 
If $(\ref{sz>m})$, $\eqref{spd_phi}$ and $(\ref{<gamma<})$ hold and the objective function $f$ is bounded below, then \eqref{dto0} holds.
Furthermore, if $\{x_k\}$ is bounded, then any accumulation point of  $\{x_k\}$ is  a critical point of $\eqref{min}$.\end{theorem}

\subsection{Semi-smooth Newton method for computing scaled proximal mappings}\label{subsec:prox}
In this section, we consider the numerical method for solving subproblem \eqref{w_prox_grad}. 
Since the structure of the subproblem is the sum of a smooth convex function and a nonsmooth convex function, 
we can use proximal gradient methods, for example. 
However, computational costs for solving such subproblems become high (especially when the dimension is large) because the dimension of the subproblem is {the} same as the original problem \eqref{min}. 
Thus, we adopt Becker et al.'s technique~\cite{SIOPT_Becker} to solve the subproblem. 

We now introduce the following theorem, which can be proved by using~\cite[Theorem 3.4]{SIOPT_Becker}, as shown in Appendix \ref{Appendix:prox}.
\begin{theorem}\label{thm:prox}
Let $\bar{x},u_1,u_2\in\mathbb{R}^{n}$, $\tau>0$,
\begin{align}
B=\tau I+u_1u_1^T-u_2u_2^T\label{MLBFGS_B2},
\end{align}
$\alpha=(\alpha_1,\alpha_2)^T$ and
\begin{align}
\zeta(\alpha)=\bar{x} - \frac{\alpha_1}{\tau} u_1 + \alpha_2 (\tau I+u_1u_1^T)^{-1} u_2.
\label{def:zeta}
\end{align}
If $u_1$ and $u_2$ are linearly independent, \eqref{MLBFGS_B2} is positive definite, and $h_1$ is proper lsc convex, then 
\begin{align}\label{prox_Bk}
{\rm Prox}_{h_1}^{B}(\bar{x}) = {\rm Prox}_{\frac{1}{\tau}h_1}(\zeta(\alpha^\ast)),
\end{align}
where the mapping $\mathcal{L}:\mathbb{R}^2\to\mathbb{R}^2$ is defined by 
\begin{align}\label{def:L}
\mathcal{L}(\alpha) 
:=
\begin{pmatrix}
u_1^T(\bar{x}  +  \alpha_2(\tau I+u_1u_1^T)^{-1}u_2-   {\rm Prox}_{\frac{1}{\tau}h_1}(\zeta(\alpha))) + \alpha_1\\
u_2^T(\bar{x} - {\rm Prox}_{\frac{1}{\tau}h_1}(\zeta(\alpha))) + \alpha_2
\end{pmatrix}
\end{align}
and $\alpha^\ast$ is a unique root of $\mathcal{L}(\alpha)=0$.
\end{theorem}
The Broyden--Fletcher--Goldfarb--Shanno (BFGS) formula (namely, \eqref{SBroyden_B} with $\phi_k=0$) can be rewritten as the form \eqref{MLBFGS_B2} with 
\begin{equation}
\tau = \tau_k,\quad 
u_1 = \sqrt{\frac{\gamma_k}{s_{k-1}^Tz_{k-1}}}z_{k-1},\quad u_2=\frac{\sqrt{\tau_k}}{\|s_{k-1}\|}s_{k-1}.\label{def:u}
\end{equation}
Therefore, we can compute $x_k^+$ in \eqref{inexact_prox} by setting in \eqref{prox_Bk}
\begin{equation}
\bar{x}=x_k-H_k(\nabla g(x_k)-\xi_k)\label{def:barx}
\end{equation}
and inexactly solving the following system of equations: 
\begin{equation}
\mbox{Find $\alpha\in\mathbb{R}^2$ such that }\mathcal{L}(\alpha)=0. \label{solve:L}
\end{equation} 
We emphasize that $\mathcal{L}$ is a two-dimensional function, so the computational costs for solving the system are expected to be very cheap. 
Theorem~\ref{thm:prox} assumes that $u_1$ and $u_2$ are linearly independent.  
If $u_1$ and $u_2$ are linearly dependent, then \eqref{MLBFGS_B2} becomes a rank-one update 
and hence we can adopt \cite[Theorem 3.8]{SIOPT_Becker}. 
Moreover, at least in our numerical experiments by using the BFGS formula (namely, \eqref{def:u}), 
{the linear independence} assumption is almost always satisfied. 
Therefore, in the remainder of this section, we suppose that  $u_1$ and $u_2$ are linearly independent. 

Hereafter, we consider how to solve system \eqref{solve:L}. 
Since the function $\mathcal{L}$  in \eqref{def:L} involves a nonsmooth term ${\rm Prox}_{\frac{1}{\tau}h_1}(\zeta(\alpha)))$, the function $\mathcal{L}$ is also nonsmooth. 
However, since $h_1$ is a proper lsc convex function, ${\rm Prox}_{\frac{1}{\tau}h_1}$ is single-valued, continuous, and nonexpansive (namely, Lipschitz continuous with the modulus 1), and thus $\mathcal{L}$ is also Lipschitz continuous. 
Moreover, in many applications, ${\rm Prox}_{\frac{1}{\tau}h_1}$ is (strongly) semi-smooth, 
and then $\mathcal{L}$ is also (strongly) semi-smooth. 
For example, ${\rm Prox}_{\frac{1}{\tau}h_1}$ is strongly semi-smooth when $h_1$ is the $\ell_1$-norm. 
Other practical regularizers are (strongly) semi-smooth (see, for example,~\cite{arXiv_Patrinos2014,JSC_Xiao2018}). 

In general, a semi-smooth Newton method~\cite{MOR_Qi1993} can be used to solve a system of semi-smooth equations. Under mild assumptions, the method converges (quadratically) superlinearly for (strongly) semi-smooth functions. 
Accordingly, we adopt the semi-smooth Newton method to solve system \eqref{solve:L}. 

To develop a semi-smooth Newton method for \eqref{solve:L}, we first consider a stopping criterion for the algorithm. 
Considering \eqref{prox_Bk} and \eqref{def:barx},  we can rewrite \eqref{inexact_prox} as 
\begin{equation}
x_k^+={\rm Prox}_{\frac{1}{\tau}h_1}(\zeta(\bar\alpha)), \label{xkp_alver}
\end{equation}
where $\bar\alpha\in\mathbb{R}^2$ is an approximate solution of \eqref{solve:L} such that  \eqref{inexact} and \eqref{sub_opt} hold. 
To define the residual $r_k$ in \eqref{inexact} and \eqref{sub_opt}, 
we give the following proposition, whose proof is given in Appendix~\ref{Appendix:prox2}.
\begin{proposition}\label{prop:res}
Suppose all assumptions of Theorem~{\rm\ref{thm:prox}} hold.
Let $U=[-u_1,u_2]\in\mathbb{R}^{n\times 2}$, $H=B^{-1}$ and $\bar{x}=x-H(\nabla g(x)- \xi)$.
Then the following holds for all $\alpha$:
\begin{equation*}\label{res_thm}
U\mathcal{L}(\alpha)
\in \nabla g(x) - \xi +B({\rm Prox}_{\frac{1}{\tau}h_1}(\zeta(\alpha))-x)
+\partial h_1({\rm Prox}_{\frac{1}{\tau}h_1}(\zeta(\alpha))). 
\end{equation*}
\end{proposition}
It follows from Proposition~\ref{prop:res}, \eqref{sub_opt}, and \eqref{xkp_alver} that we can regard $U\mathcal{L}(\alpha)$ as the residual.  

To guarantee the global convergence, we define the following merit function:
\[
\Psi(\alpha)=\frac12 \|\mathcal{L}(\alpha)\|^2, 
\]
and adopt a standard line search technique.
Summarizing the above arguments, we give Algorithm~\ref{alg:semismoothNewton}.
\begin{algorithm}
\caption{semi-smooth Newton method with line search}
\label{alg:semismoothNewton}
\begin{algorithmic}
\Require{$\alpha_0\in \mathbb{R}^2$, $x_k,u_1,u_2\in\mathbb{R}^n$, $\tau>0$, $\sigma\in(0,1/2),\ \rho\in(0,1)$, $\theta_k\in[\bar{\theta},1]$, 
$\varepsilon>0$
}
\State $B_k \leftarrow \tau I + u_1u_1^T - u_2u_2^T$
\State $H_k \leftarrow B_k^{-1}$
\State $\bar{x} \leftarrow x_k -H_k\nabla g(x_k)$
\State $U \leftarrow [-u_1, u_2]$
\For{$j = 0, 1, 2, ...$}
\State $x_k^+ \leftarrow {\rm Prox}_{\frac{1}{\tau}h_1}\left(\zeta(\alpha_j)\right)$
\State $r_k \leftarrow U\mathcal{L}(\alpha_j)$
\State $d_k \leftarrow x_k^+ - x_k$
\If{
either condition \eqref{inexact} or $\|d_k\|\le \varepsilon$ is satisfied
} \State stop. \EndIf
\State Select 
$V_j\in \partial^C \mathcal{L}(\alpha_j)$.
\State $p_j \leftarrow - V_j^{-T}\mathcal{L}(\alpha_j)$
\State $l \leftarrow 0$
\While{condition \begin{align}
\Psi(\alpha_j+\rho^{l}p_j)\le (1-2\sigma \rho^{l})\Psi(\alpha_j) \label{inner_armijo}
\end{align} 
\State is not satisfied}
\State $l \leftarrow l+1$
\EndWhile
\State $t_j \leftarrow \rho^l$
\State $\alpha_{j+1} \leftarrow \alpha_j + t_j p_j$
\EndFor
\end{algorithmic}
\end{algorithm}

\begin{remark}
Note that $\varepsilon$ in Algorithm~\ref{alg:semismoothNewton} is 
the constant appearing in Algorithm~\ref{alg:proximal-DC-Newton}. 
Thus, if Algorithm~\ref{alg:semismoothNewton} is stopped by $\|d_k\|\le \varepsilon$, then Algorithm~\ref{alg:proximal-DC-Newton} is also stopped. 
Otherwise, we have $\|d_k\|> \varepsilon$ holds for all $j$. 
It follows from Proposition~\ref{prop:res} and \eqref{inexact} that the stopping condition of Algorithm~\ref{alg:semismoothNewton} can be rewritten by 
\begin{align}
\|U\mathcal{L}(\alpha_j)\|_{H_k}\le (1-\theta_k)\|{\rm Prox}_{\frac{1}{\tau}h_1}\left(\zeta(\alpha_j)\right)-x_k\|_{B_k}
=(1-\theta_k)\|d_k\|_{B_k}
. \label{inner:stop}
\end{align}
Thus, it suffices to show $\lim_{j\to \infty}\alpha_j=\alpha^\ast$ ($\alpha^\ast$ is the unique solution of $\mathcal{L}(\alpha)=0$), instead of \eqref{inner:stop}. 
\end{remark}
Next, we consider the global convergence properties for Algorithm~\ref{alg:semismoothNewton}. 
There are many studies on global convergence properties for semi-smooth Newton methods with line search 
under the assumption that the merit function $\Psi$ is continuously differentiable (see \cite{facchinei2003finite,QiSu93,QSZ00,QS1998} for example).  
However, to the best of our knowledge, there are not many studies on a global convergence property for the nondifferentiable case.  
Thus, we provide the proofs for the global convergence of the algorithm in Appendix \ref{Appendix:global_semismooth}.
\begin{theorem}\label{thm_global_semismooth}
Consider Algorithm~{\rm\ref{alg:semismoothNewton}}. 
Suppose that all assumptions of Theorem~{\rm \ref{thm:prox}} hold 
and ${\rm Prox}_{\frac{1}{\tau}h_1}$ is directionally differentiable. 
In addition, assume that the level set  at the initial point: 
\[
\mathcal{S}_0=\{
\alpha \mid \Psi(\alpha) \le \Psi(\alpha_0)
\}
\]
is bounded and any element of $\partial^C \mathcal{L}(\alpha)$ is nonsingular for any $\alpha\in \mathcal{S}_0$. 
If the condition
\begin{equation}
\label{eq:descent_direction}
\Psi^\prime (\alpha_j;p_j)\le (V_j\mathcal{L}(\alpha_j))^Tp_j 
\end{equation}
holds for all $j$, then  the sequence $\{\alpha_j\}$ generated by Algorithm~{\rm\ref{alg:semismoothNewton}} either terminates at the unique solution $\alpha^\ast$ of \eqref{solve:L} or converges to $\alpha^\ast$.
\end{theorem}
As mentioned above, in many applications, ${\rm Prox}_{\frac{1}{\tau}h_1}$ is semi-smooth. 
Because a semi-smooth function is directionally differentiable, 
the assumption of the directional differentiability of ${\rm Prox}_{\frac{1}{\tau}h_1}$ is reasonable. 

Since the function $\|\cdot\|^2$ is continuously differentiable and $\mathcal{L}$ is Lipschitz continuous, 
it follows from \cite[Proposition 7.1.11]{facchinei2003finite} that 
\begin{align*}
\partial^C\Psi(\alpha)= \{V\mathcal{L}(\alpha)\mid V\in\partial^C \mathcal{L}(\alpha)\}. 
\end{align*}
Therefore, any element of $\partial^C\Psi(\alpha)$ can be expressed by the form $V\mathcal{L}(\alpha)$ for some $V\in\partial^C \mathcal{L}(\alpha)$, and conversely 
$V\mathcal{L}(\alpha)\in\partial^C\Psi(\alpha)$ holds for any $V\in\partial^C \mathcal{L}(\alpha)$. 
Thus, it follows from \cite[Proposition 7.1.17]{facchinei2003finite} that there exists $V_j\in \partial^C \mathcal{L}(\alpha_j)$ such that $\Psi^\prime(\alpha_j;p_j)=(V_j\mathcal{L}(\alpha_j))^Tp_j$, which yields \eqref{eq:descent_direction}. 
Though it is not obvious how to choose $V_j$ satisfying \eqref{eq:descent_direction} in practice, 
in our numerical experiments reported in Section~\ref{sec:numerical}, there was no case where  condition \eqref{eq:descent_direction}  was violated. 

We now introduce local convergence properties of Algorithm~\ref{alg:semismoothNewton}. 
Although the proof is almost the same as \cite{QiSu93,QSZ00}, we provide the proof in Appendix \ref{Appendix:local_semismooth} for the readability.
\begin{theorem}\label{thm:conv}
Assume that all assumptions of Theorem~{\rm\ref{thm_global_semismooth}} hold 
and ${\rm Prox}_{\frac{1}{\tau}h_1}$ is semi-smooth. 
If Algorithm~{\rm\ref{alg:semismoothNewton} generates an infinite sequence $\{\alpha_j\}$, }
then $\{\alpha_j\}$ converges to the unique solution $\alpha^\ast$ of \eqref{solve:L} Q-superlinearly. 
Moreover, if ${\rm Prox}_{\frac{1}{\tau}h_1}$ is strongly semi-smooth, then  $\{\alpha_j\}$ converges to the solution $\alpha^\ast$ Q-quadratically. 
\end{theorem}

In Theorem~\ref{thm_global_semismooth}, we assume the boundedness of the level set at the initial point. 
We now consider a sufficient condition to guarantee this assumption for any initial point $\alpha_0$. 
For this purpose, we restrict the approximate matrix to the BFGS formula, namely \eqref{def:u}. 
The proof of the following proposition is given in Appendix~\ref{sec:A6}. 
\begin{proposition}\label{prop2pprime}
Let $\tau$, $u_1$ and $u_2$ be given in \eqref{def:u}. 
Suppose that Assumption~\ref{ass:Lip} and conditions \eqref{sz>m} and \eqref{<gamma<} hold.  
Moreover, assume that 
$u_1$ and $u_2$ are linearly independent and 
there exists a positive constant $\bar{c}$ such that 
\begin{align}\label{ass:bound_cdiff}
\|v\|\le\bar c\qquad \forall v\in\partial h_1(x)
\end{align}
for any $x \in {\rm dom}\, h_1=\{x \mid h_1(x)<\infty\}$. 
Then, the function $\Psi$ is coercive, namely, the following holds: 
\begin{align*}
\lim_{\|\alpha\|\to\infty}\Psi (\alpha)=\infty. 
\end{align*}
\end{proposition}
For example, if $h_1(x)=\lambda\|x\|_1$ ($\lambda>0$), then $\partial h_1(x)\subset [-\lambda,\lambda]^n$, 
and hence $\partial h_1(x)$ is bounded for any $x\in\mathbb{R}^n$.
Thus, condition~\eqref{ass:bound_cdiff} holds for a typical class of regularizer.

\section{Numerical experiments}\label{sec:numerical}
In this section, we investigate the numerical performance of Algorithm \ref{alg:proximal-DC-Newton}. 
We test least squares problems with the $\ell_{1-2}$ regularizer in Section~\ref{sec:l12} and with the log-sum penalty in Section~\ref{sec:LSP}. 
All the numerical experiments were performed in MATLAB 2019b on a PC with 2 GHz Quad-Core Intel Core i5 and 16GB RAM running macOS Catalina.

\subsection{Least squares problems with $\ell_{1-2}$ regularizer}\label{sec:l12}
We consider the least squares problems with the $\ell_{1-2}$ regularizer~\cite{yin2015minimization}:
\begin{equation}\label{LSQ12}
\min_{x\in\mathbb{R}^n} \frac12\|Ax-b\|^2 + \lambda\|x\|_1 - \lambda\|x\|,
\end{equation}
where $A\in\mathbb{R}^{m\times n}$, $b\in\mathbb{R}^m$, and $\lambda>0$ is a regularization parameter.

To solve \eqref{LSQ12}, we test the six methods given in Table \ref{tb:methodslist1}. 
In mBFGS(S-Newton) and mBFGS(V-FISTA), we use the memoryless BFGS formula, which is \eqref{SBroyden_B} with $\phi_k=0$, $\gamma_k=\frac{s_{k-1}^Tz_{k-1}}{z_{k-1}^Tz_{k-1}}$, $\tau_k=1$ and {\eqref{nu} with $\tilde\nu=10^{-6}$}. 
Note that these parameters were used in~\cite{COAP_Nakayama2021}
\footnote{
Conditions \eqref{sz>m} and \eqref{<gamma<} hold with $\underline\nu=10^{-6}$, $\bar\nu=L+10^{-6}$, $\underline\gamma=\frac{\underline\nu}{(L+\bar\nu)^2}$ and  $\overline\gamma=\frac{1}{\underline\nu}$.}. 
In mSR1(V-FISTA), we use the memoryless SR1 formula, which is \eqref{SBroyden_B} with $\phi_k=\frac{\gamma_ks_{k-1}^Tz_{k-1}}{(\gamma_kz_{k-1}-s_{k-1})^Ts_{k-1}}$, $\gamma_k=0.8\frac{s_{k-1}^Tz_{k-1}}{z_{k-1}^Tz_{k-1}}$, $\tau_k=1$ and \eqref{nu}. 
In L-BFGS(TFOCS), we use  the limited memory BFGS method~\cite{nocedal1980updating,nocedal2006numerical} as $B_k$.
For the line search in Algorithm~\ref{alg:proximal-DC-Newton}, we set $\delta=0.5$ and $\beta_k=0.5$. 
To solve the subproblem \eqref{inexact_prox} in mBFGS(S-Newton), we use Algorithm~\ref{alg:semismoothNewton} with $\sigma=10^{-4}$, $\rho=0.5$, 
and $\alpha_0=(0,0)^T$ and we set $\theta_k=0.99$.  
We choose \eqref{eq:jac} in Appendix~\ref{sec:Adirectional} as $V_j$ in Algorithm~\ref{alg:semismoothNewton}. 
As mentioned in Section~\ref{subsec:prox}, there was no case where  condition \eqref{eq:descent_direction}  was violated. 
For mBFGS(V-FISTA) and mSR1(V-FISTA), we use Variant-FISTA (V-FISTA) and set $\theta_k=0.1$, as in \cite{COAP_Nakayama2021}. 
For L-BFGS(TFOCS), we use the Templates for First-Order Conic Solvers (TFOCS)~\cite{becker2011templates}, which is a well-known software for solving convex programming. Here, pDCAe is the proximal DCA with extrapolation proposed by Wen et al.~\cite{COAP_Wen2018} and we use the same parameters as~\cite{COAP_Wen2018}\footnote{In pDCAe, the constant $L$ in \eqref{Lip_Ass} is computed via the MATLAB code ``{\rm L=norm(A*A')}"; when $m\leq2000$, and by ``{\rm opts.issym = 1; L= eigs(A*A',1,'LM',opts);}" otherwise.}. 
The nmAPG approach is the nonmonotone accelerated proximal gradient method proposed by Li and Lin~\cite{NIPS_Li2015}\footnote{We implement Algorithm 4 in the supplemental of \cite{NIPS_Li2015}.}, which is a well-known efficient proximal gradient-type method for nonconvex functions. Note that mSR1(V-FISTA) corresponds to the method of Liu and Takeda~\cite{OptOnline_Liu2021}, and L-BFGS(TFOCS) corresponds to 
a DCA version of the proximal Newton-type method~\cite{SIOPT_Lee2014}, although they are slightly different. 
For all methods and problems, the initial point $x_0\in\mathbb{R}^n$ was set as the zero vector. The stopping conditions were \[\|x_k^+-x_k\|\leq 10^{-5}\max\{1,\|x_k\|\}\] for Algorithm~\ref{alg:proximal-DC-Newton}, and $\|x_{k+1}-x_k\|\leq 10^{-5}\max\{1,\|x_k\|\}$ for the other tested methods.

\begin{table}
	\caption{Tested methods}
	\centering
\begin{tabular}{c|l|c}\hline
Method name & Algorithm & How to solve \eqref{inexact_prox}\\
\hline\hline
mBFGS(S-Newton) & Algorithm~\ref{alg:proximal-DC-Newton} with memoryless BFGS formula & Algorithm~\ref{alg:semismoothNewton} \\
mBFGS(V-FISTA) & Algorithm~\ref{alg:proximal-DC-Newton} with memoryless BFGS formula & V-FISTA~\cite{COAP_Nakayama2021}\\
mSR1(V-FISTA) & Algorithm~\ref{alg:proximal-DC-Newton} with memoryless SR1 formula & V-FISTA~\cite{COAP_Nakayama2021} \\
L-BFGS(TFOCS) & Algorithm~\ref{alg:proximal-DC-Newton} with limited memory BFGS method& TFOCS~\cite{becker2011templates} \\
pDCAe & proximal DCA with extrapolation~\cite{COAP_Wen2018}  & - \\
nmAPG &  nonmonotne accelerated proximal gradient method~\cite{NIPS_Li2015} &- \\
\hline
\end{tabular}\label{tb:methodslist1}
\end{table}

For $A$ and $b$ in \eqref{LSQ12}, we generate a matrix and a vector randomly following Wen et al.~\cite{COAP_Wen2018}: 
(i) We generate a matrix $A$ with independent and identically distributed (i.i.d.) standard Gaussian entries, and then normalize this matrix so that the columns of $A$ have unit norms. 
(ii) A subset $T$ of size $p$ is then chosen uniformly at random from $\{1, 2, 3, \cdots, n\}$ and a $p$-sparse vector $\hat{x}\in\mathbb{R}^n$ with i.i.d. standard Gaussian entries on $T$ is generated. 
(iii) We set $b = A\hat{x}+0.01u,$ where $u\in\mathbb{R}^m$ is a random vector with i.i.d. standard Gaussian entries.

We consider $(m,n,p)=(720l,2560l,80l)$ for $l=1,2,...,5$. For each triple $(m,n,p)$, we generate $20$ instances randomly according to the above steps. 

Fig.~\ref{fig:time-L12} and \ref{fig:iter-L12}, respectively, show the average central processing unit (CPU) time and average number of iterations for each $l$ with $\lambda=0.01$ (top-left), $\lambda=0.005$ (top-right), $\lambda=0.001$ (bottom-left) and $\lambda=0.0005$ (bottom-right). 
In Fig.~\ref{fig:iter-L12}, we use the same markers as in Fig.~\ref{fig:time-L12}, so we omit the legend.

For all cases, mBFGS(S-Newton) was superior to or at least comparable with the other methods from the viewpoint of CPU time and the number of iterations. On the other hand, mSR1(V-FISTA) was comparable with mBFGS(S-Newton) for $\lambda=0.01$, but the performance of mSR1 (V-FISTA) deteriorated slightly as $\lambda$ decreased. For mBFGS(V-FISTA), the number of iterations tended to increase as $\lambda$ decreased. L-BFGS(TFOCS) had the lowest number of iterations, but the worst CPU time due to the high computing costs. For pDCAe, the CPU time was comparable with that of mBFGS(S-Newton) for $\lambda=0.001$ and $0.005$. However, this method deteriorated as $\lambda$ became large, and the number of iterations was high for all cases. 
For nmAPG, the performance was in the middle of all methods. 
Summarizing the results, the numerical experiments showed the effectiveness of Algorithm~\ref{alg:proximal-DC-Newton} with Algorithm~\ref{alg:semismoothNewton}. 

\begin{figure}[h]
  \begin{center}
    \begin{tabular}{c}
      \begin{minipage}{0.45\hsize}
        \begin{center}
          \includegraphics[width=.95\linewidth]{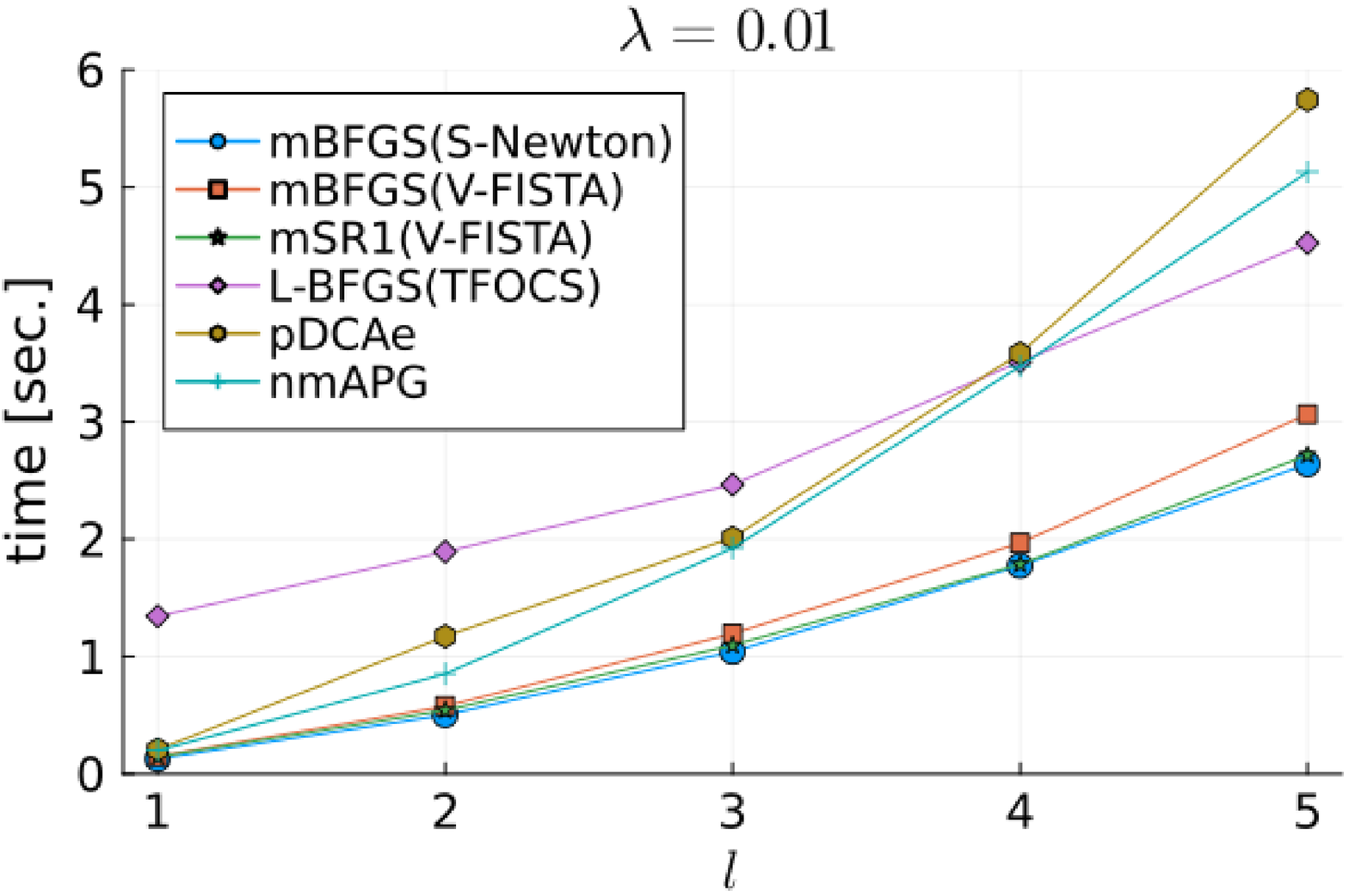}
                  \end{center}
      \end{minipage}
      \begin{minipage}{0.45\hsize}
        \begin{center}
          \includegraphics[width=.95\linewidth]{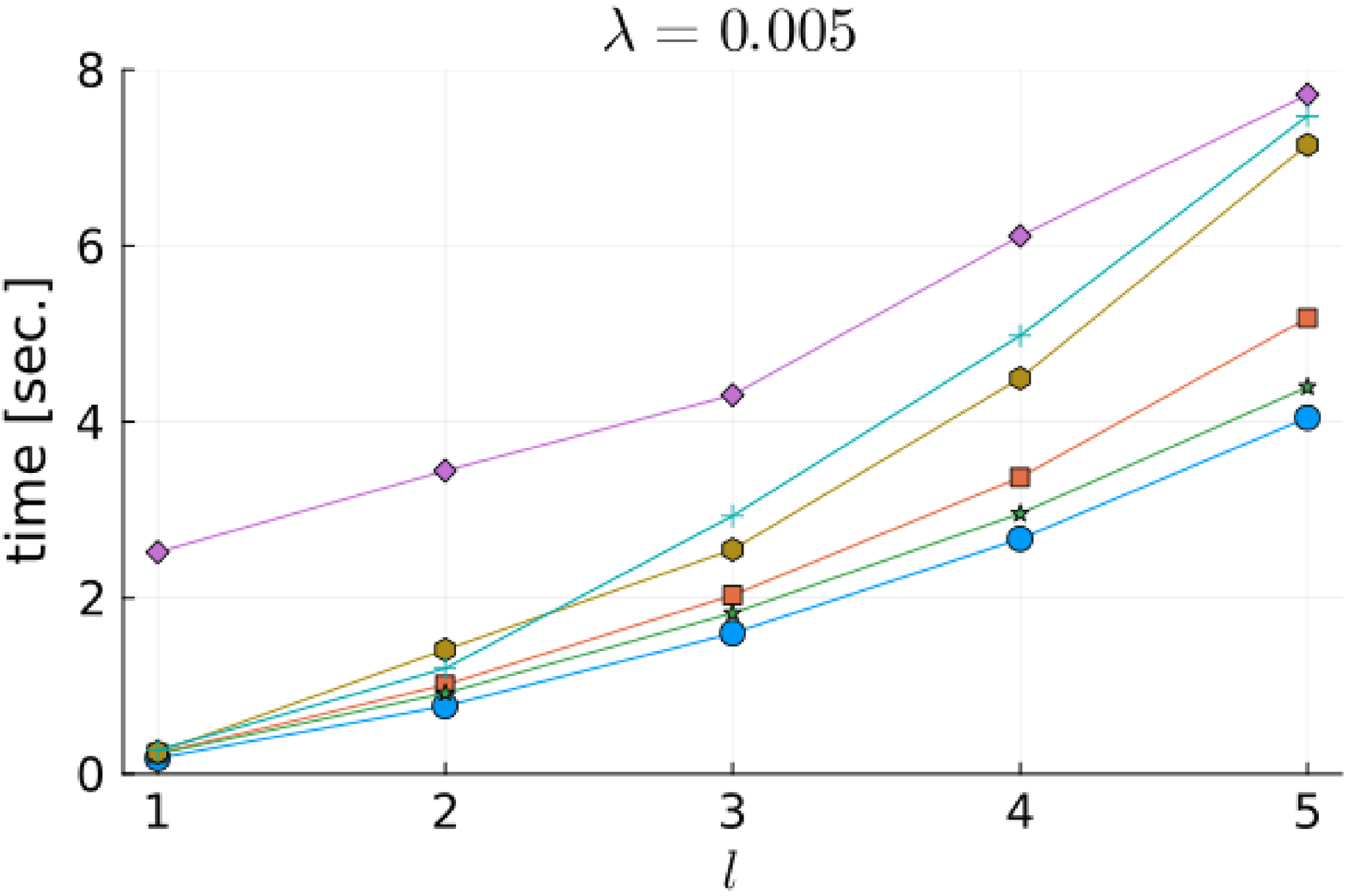}
                  \end{center}
      \end{minipage}
      \\
      \begin{minipage}{0.45\hsize}
        \begin{center}
          \includegraphics[width=.95\linewidth]{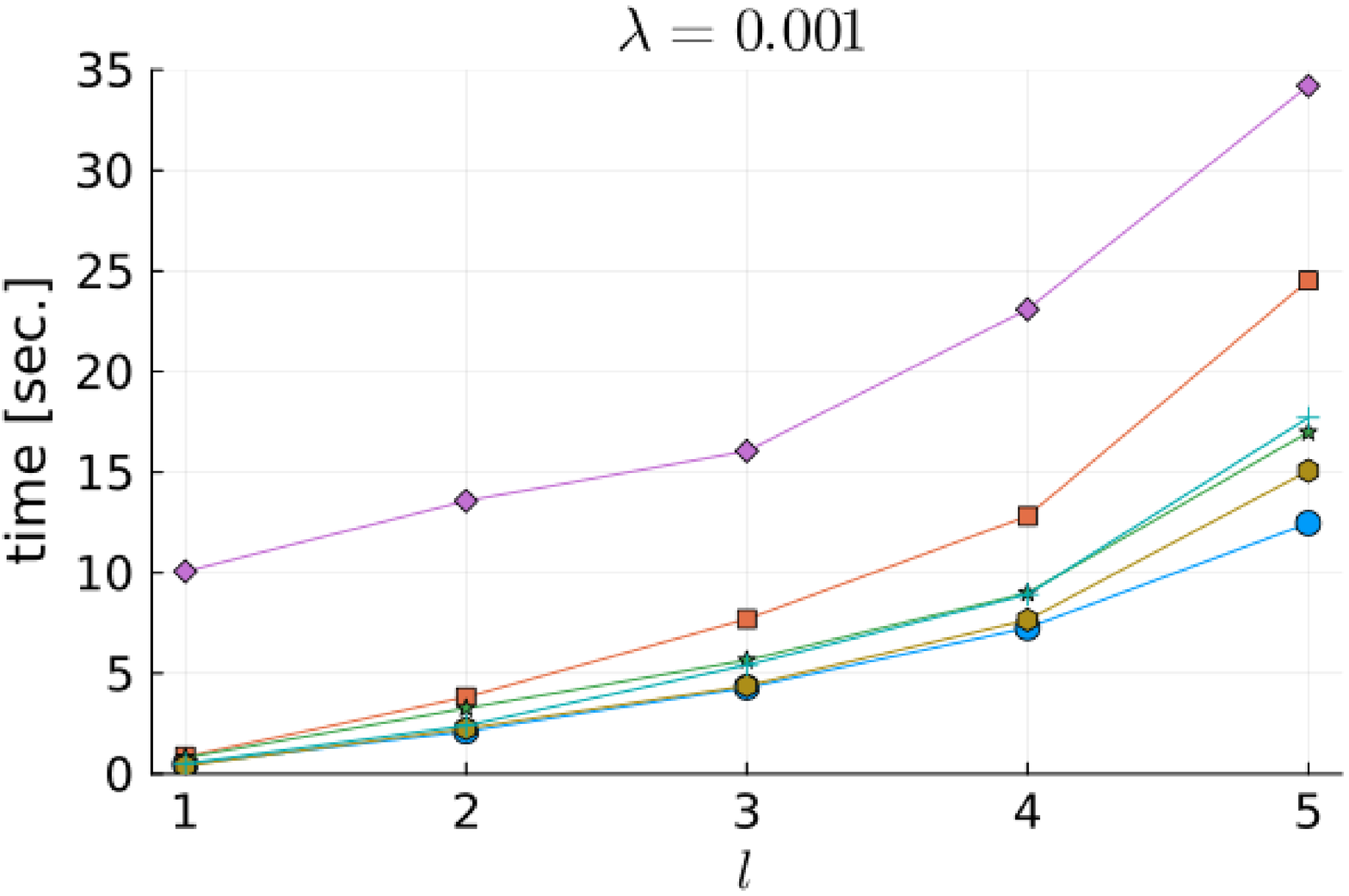}
                  \end{center}
      \end{minipage}
      \begin{minipage}{0.45\hsize}
        \begin{center}
          \includegraphics[width=.95\linewidth]{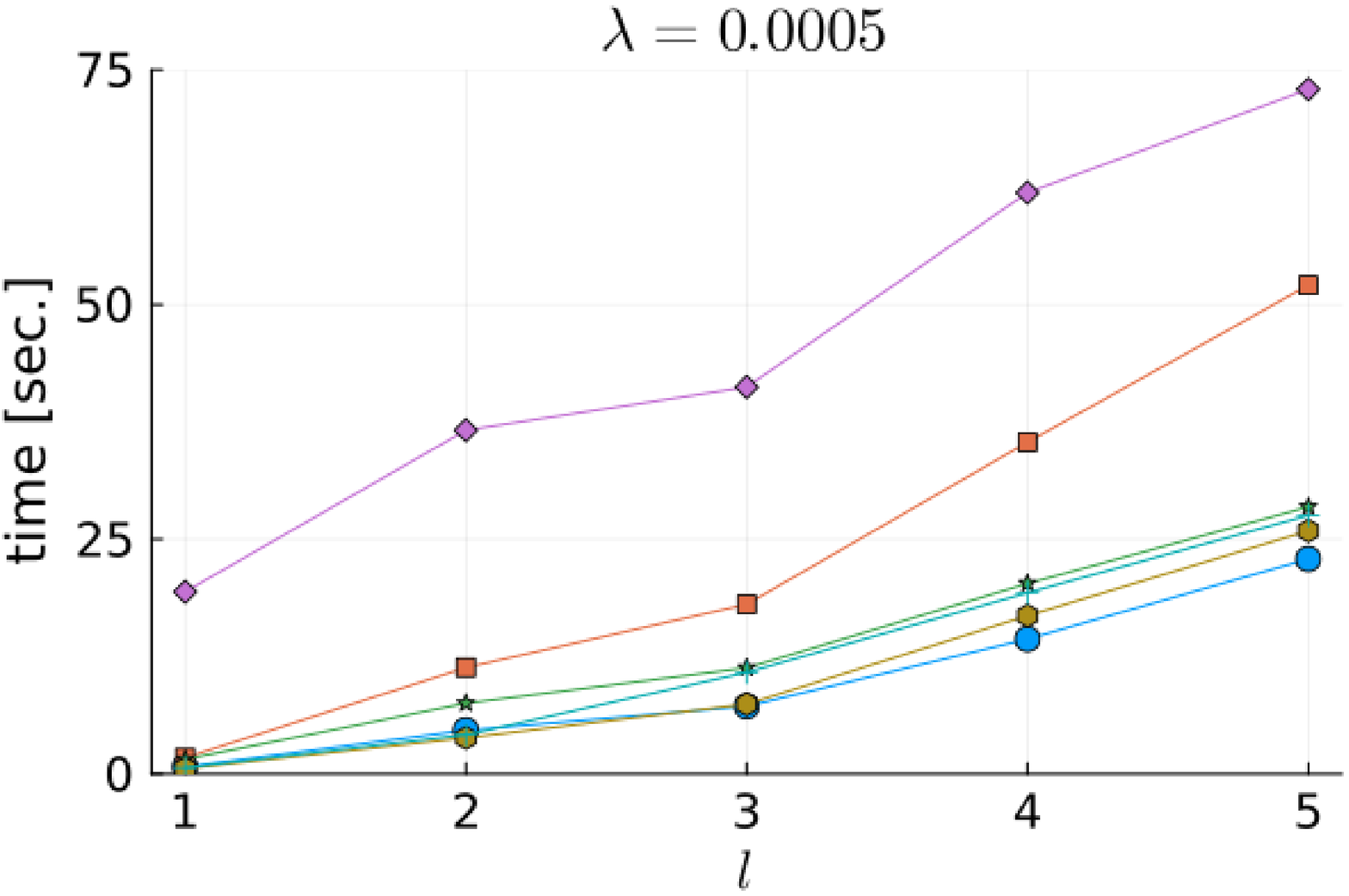}
                  \end{center}
      \end{minipage}
    \end{tabular}
    \caption{Average CPU time to solve \eqref{LSQ12}}
    \label{fig:time-L12}
  \end{center}
\end{figure}

\begin{figure}[h]
  \begin{center}
    \begin{tabular}{c}
      \begin{minipage}{0.45\hsize}
        \begin{center}
          \includegraphics[width=.95\linewidth]{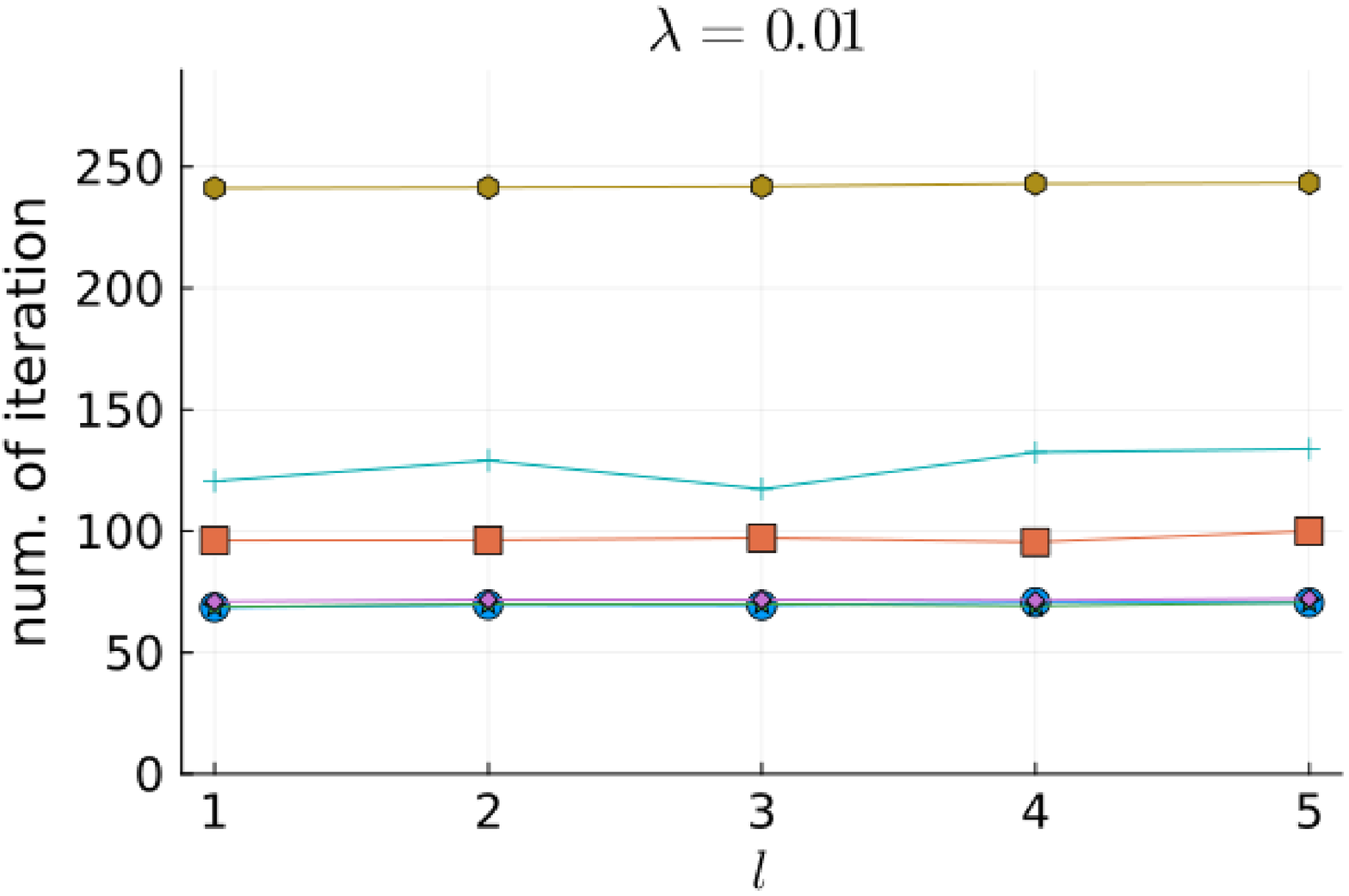}
                  \end{center}
      \end{minipage}
      \begin{minipage}{0.45\hsize}
        \begin{center}
          \includegraphics[width=.95\linewidth]{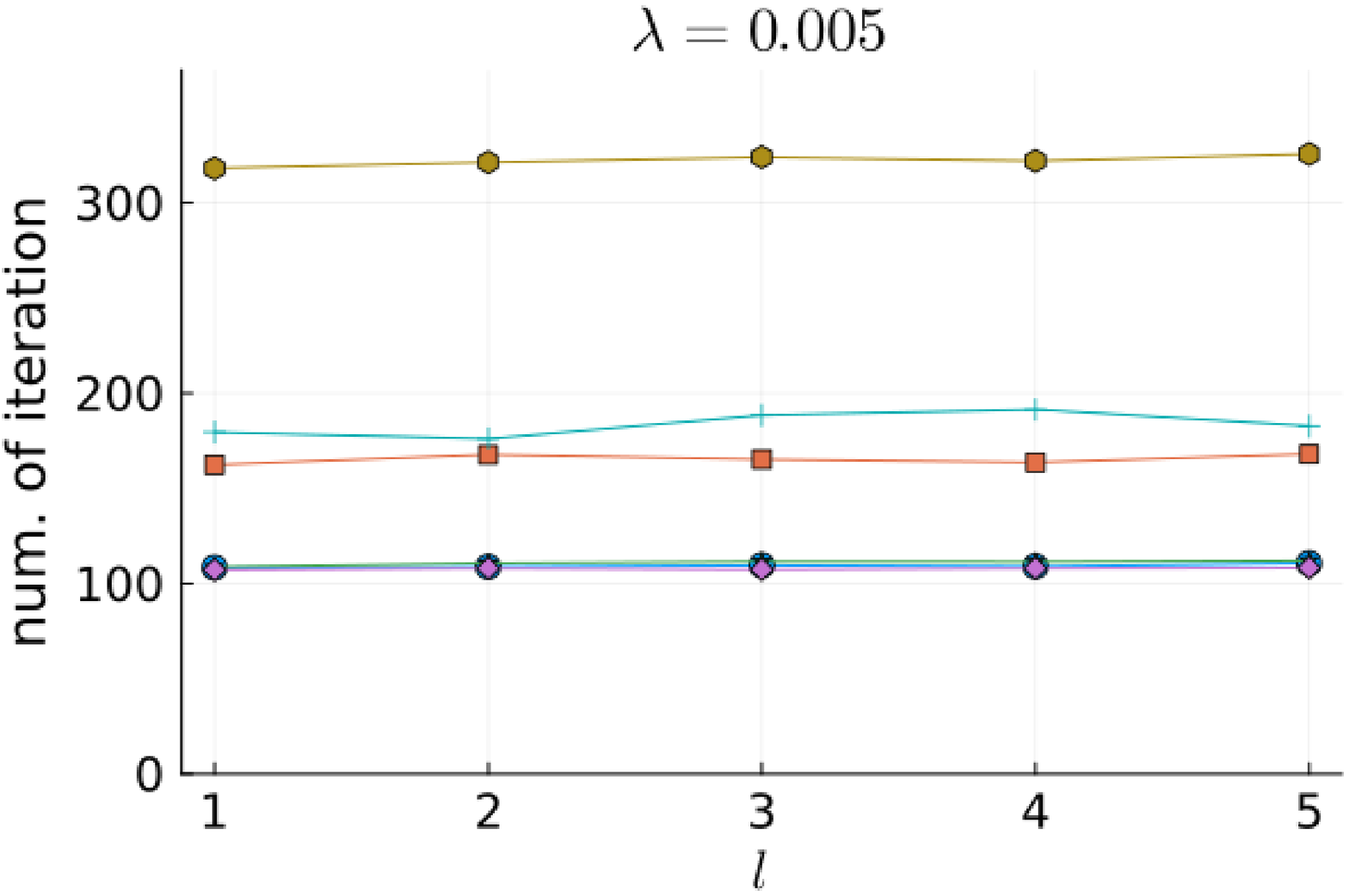}
                  \end{center}
      \end{minipage}
      \\
      \begin{minipage}{0.45\hsize}
        \begin{center}
          \includegraphics[width=.95\linewidth]{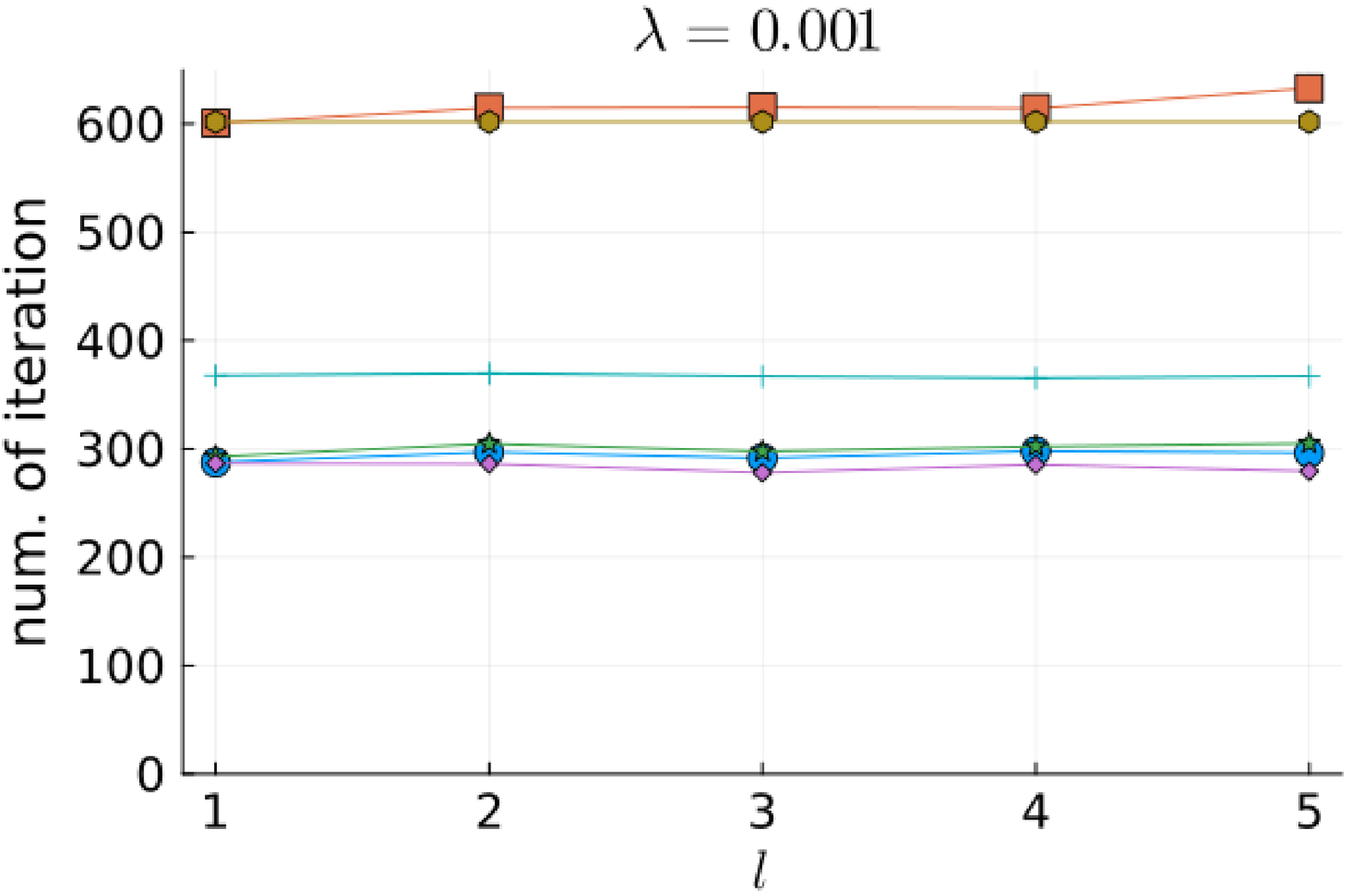}
                  \end{center}
      \end{minipage}
      \begin{minipage}{0.45\hsize}
        \begin{center}
          \includegraphics[width=.95\linewidth]{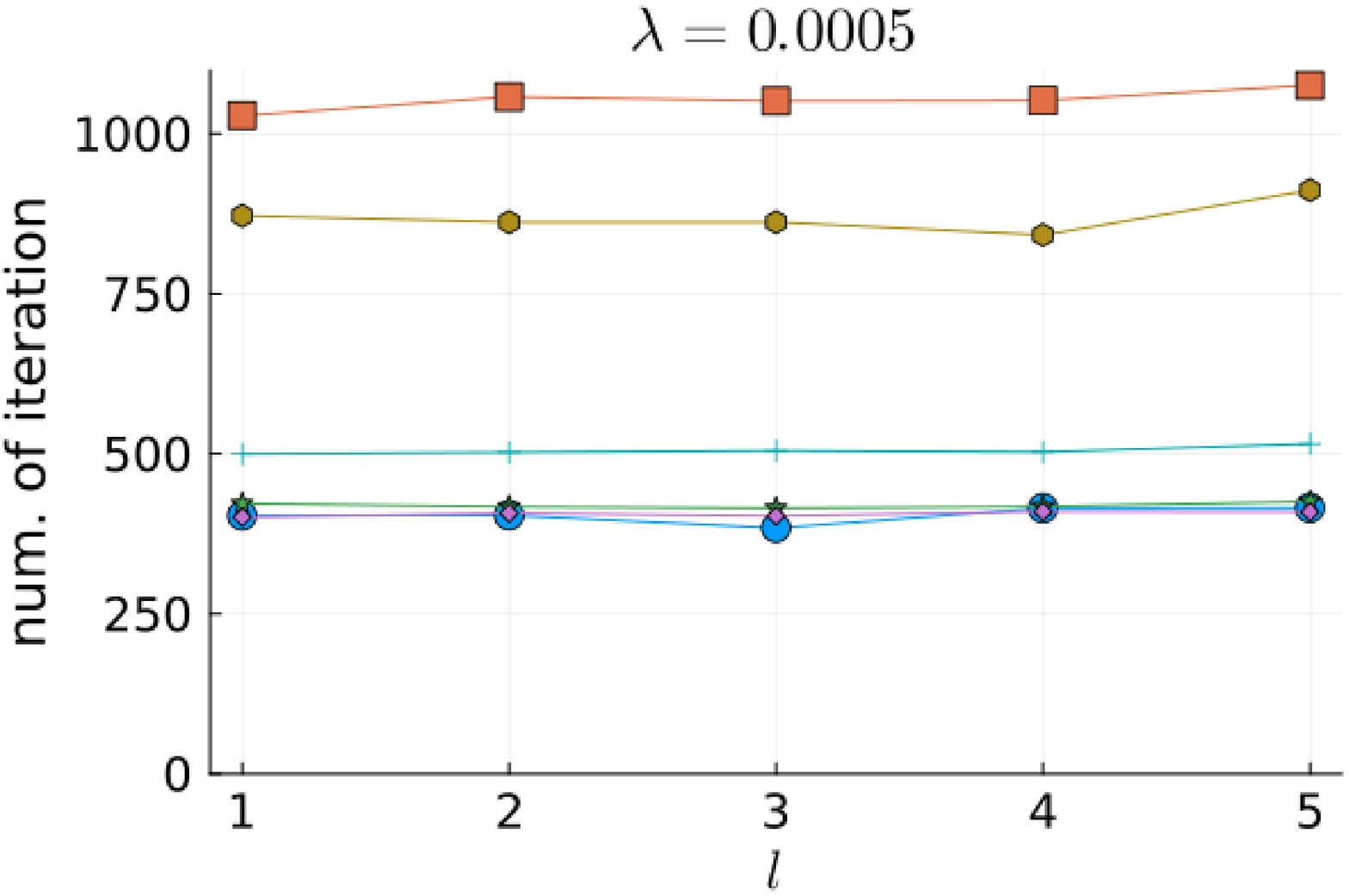}
                  \end{center}
      \end{minipage}
    \end{tabular}
    \caption{Average number of iterations to solve \eqref{LSQ12}}
    \label{fig:iter-L12}
  \end{center}
\end{figure}

\subsection{Least squares problems with log-sum penalty}\label{sec:LSP}
We consider the least squares problems with the log-sum penalty~\cite{candes2008enhancing}:
\begin{equation}\label{LSP}
\min_{x\in\mathbb{R}^n} \frac12\|Ax-b\|^2 + \lambda\sum_{i=1}^n\log\left({1+\frac{\lvert(x)_i\rvert}{\epsilon}}\right),
\end{equation}
where $A\in\mathbb{R}^{m\times n}$, $b\in\mathbb{R}^m$, $\lambda>0$ is a regularization parameter, and $\epsilon$ is a parameter. 
Since this problem can be rewritten as 
\[
\min_{x\in\mathbb{R}^n} \underbrace{\frac12\|Ax-b\|^2}_{g(x)} + \underbrace{\frac\lambda\epsilon\|x\|_1}_{h_1(x)} - \underbrace{\lambda\sum_{i=1}^n\left(\frac{\lvert (x)_i\rvert}{\epsilon} - \log\left(\lvert(x)_i\rvert+\epsilon\right) + \log \epsilon \right)}_{h_2(x)},
\]
we can adopt Algorithm~\ref{alg:proximal-DC-Newton}.

In this subsection, we generate $A$ and $b$ as in Section~\ref{sec:l12}, and set $\epsilon=0.5$. We test with the same settings as in Section~\ref{sec:l12}. Since mBFGS(V-FISTA) and L-BFGS(TFOCS) performed poorly in preliminary experiments, we omit these methods.

Fig.~\ref{fig:time-LSP} and \ref{fig:iter-LSP}, respectively, show average CPU time and average number of iterations for each $l$ with $\lambda=0.01$ (top-left), $\lambda=0.005$ (top-right), $\lambda=0.001$ (bottom-left) and $\lambda=0.0005$ (bottom-right). 
Fig.~\ref{fig:iter-LSP} uses the same markers as Fig.~\ref{fig:time-LSP}, so we omit the legend.

These experiments have the same tendencies as Section~\ref{sec:l12}. 
For all cases, mBFGS(S-Newton) was superior to the other methods from the perspectives of CPU time and the number of iterations. 
Though the number of iterations was almost the same mBFGS(S-Newton) and mSR1(V-FISTA), mBFGS(S-Newton) had better CPU time.  Thus, this implies that Algorithm~\ref{alg:semismoothNewton} is efficient.

\begin{figure}[h]
  \begin{center}
    \begin{tabular}{c}
      \begin{minipage}{0.45\hsize}
        \begin{center}
          \includegraphics[width=.95\linewidth]{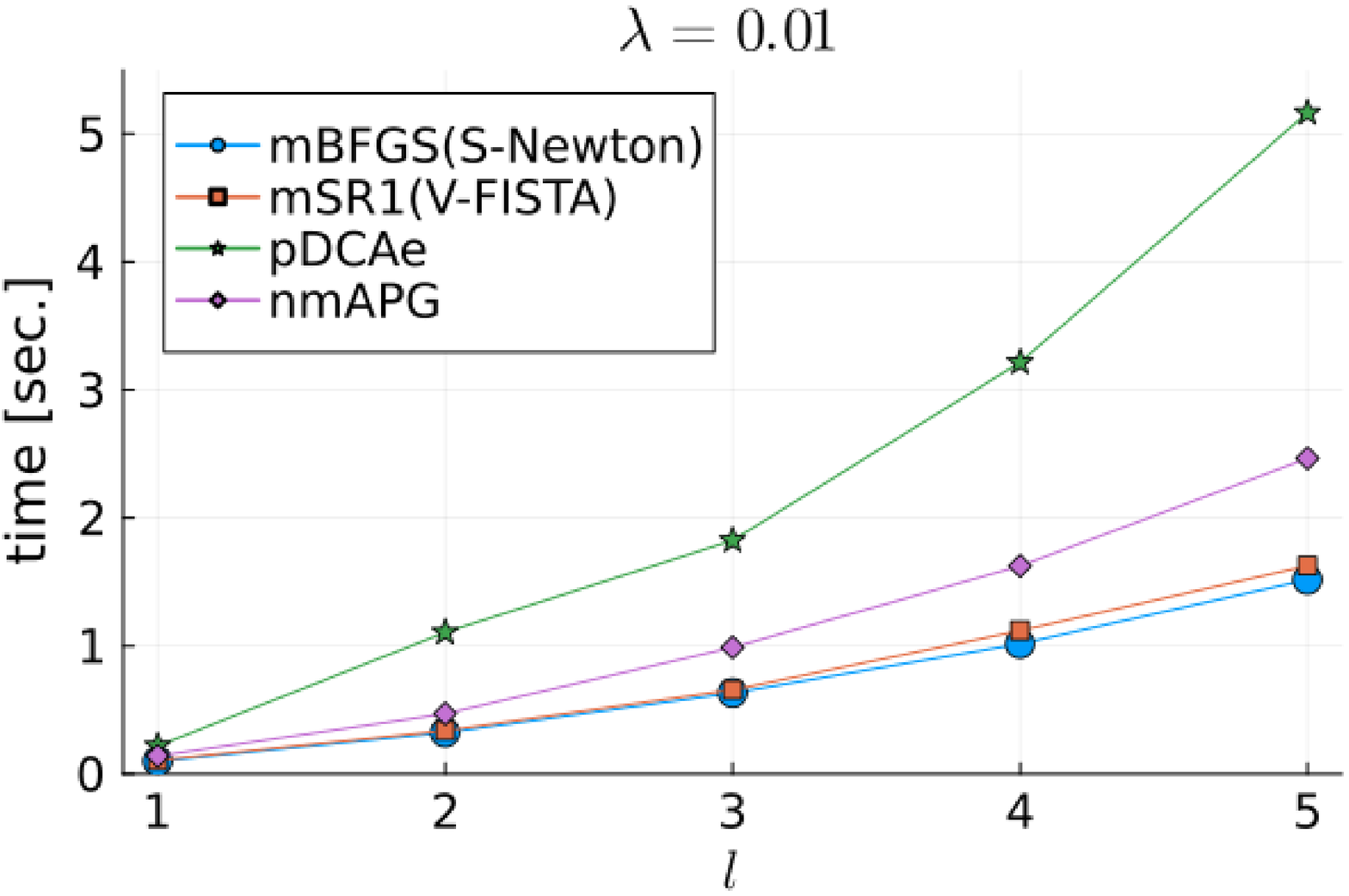}
                  \end{center}
      \end{minipage}
      \begin{minipage}{0.45\hsize}
        \begin{center}
          \includegraphics[width=.95\linewidth]{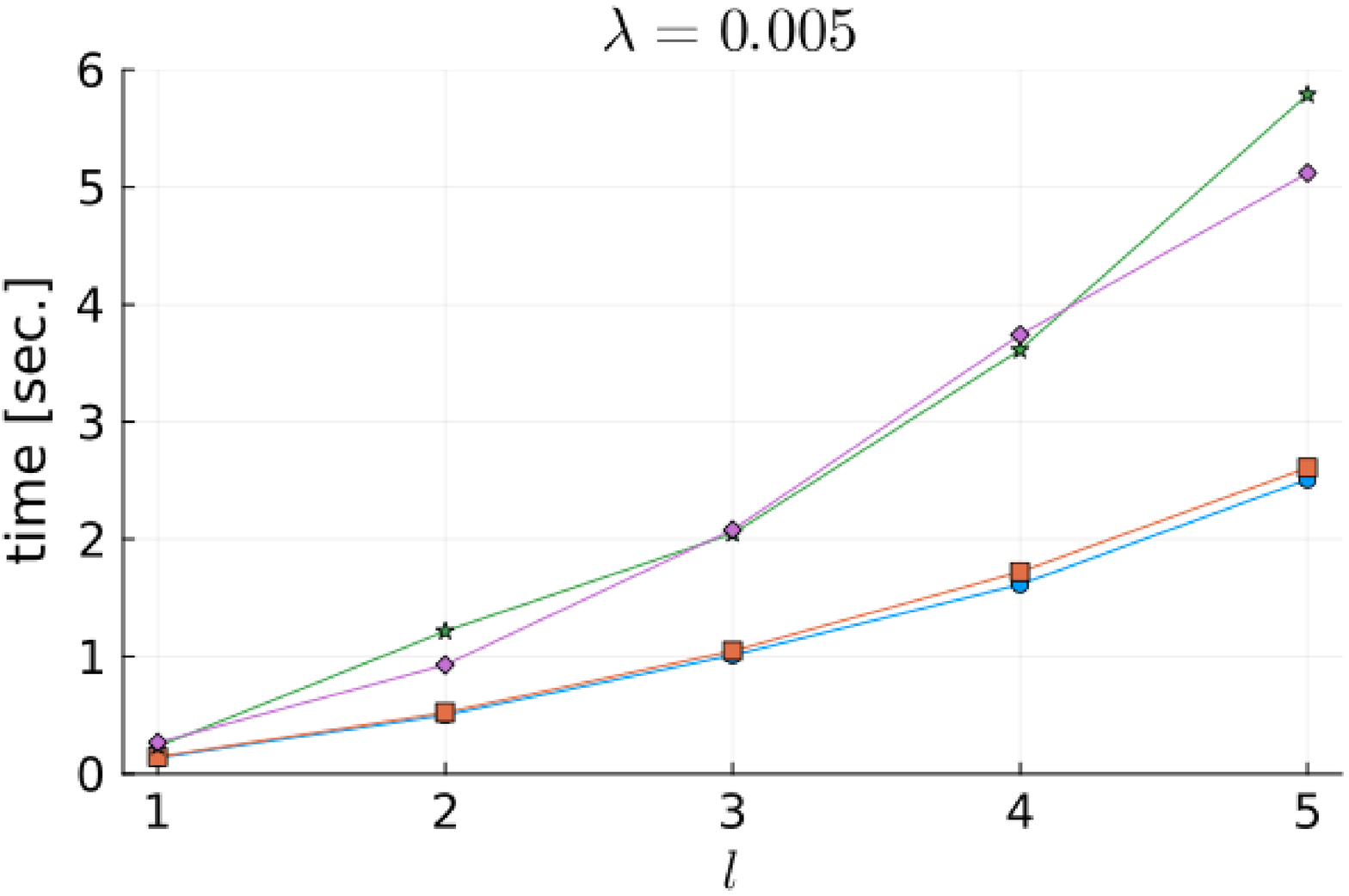}
                  \end{center}
      \end{minipage}
      \\
      \begin{minipage}{0.45\hsize}
        \begin{center}
          \includegraphics[width=.95\linewidth]{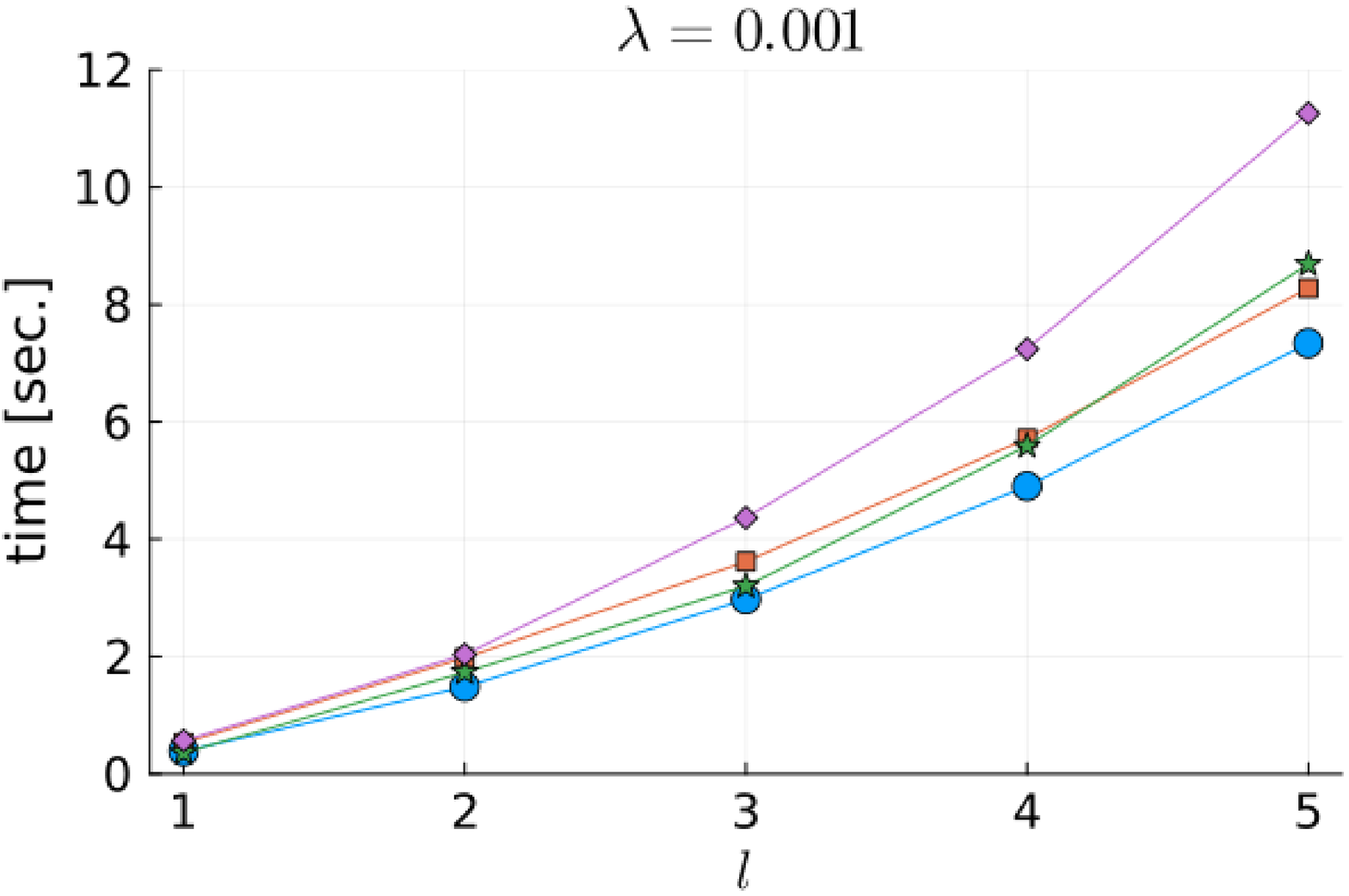}
                  \end{center}
      \end{minipage}
      \begin{minipage}{0.45\hsize}
        \begin{center}
          \includegraphics[width=.95\linewidth]{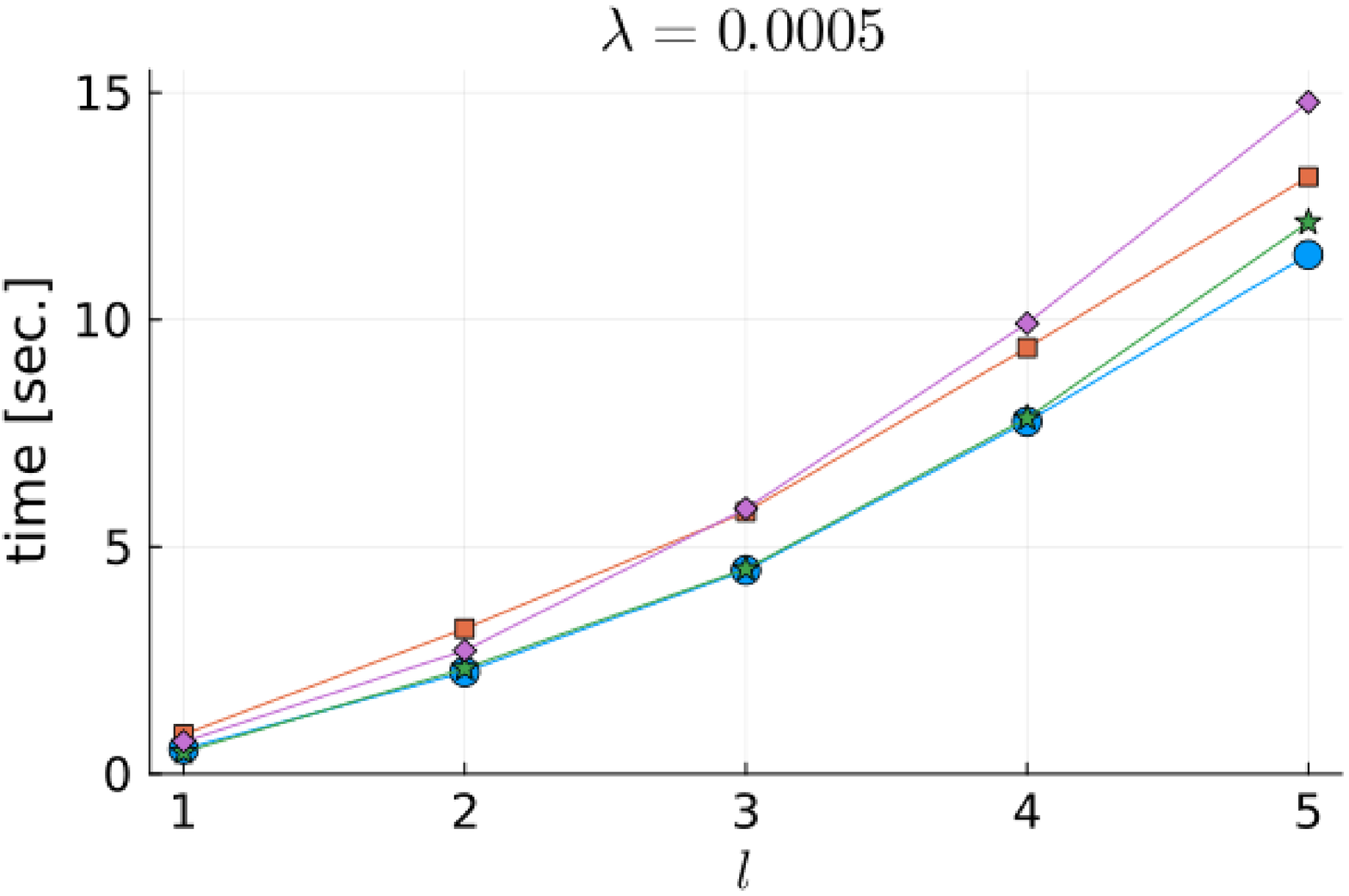}
                  \end{center}
      \end{minipage}
    \end{tabular}
    \caption{Average CPU time to solve \eqref{LSP}}
    \label{fig:time-LSP}
  \end{center}
\end{figure}

\begin{figure}[h]
  \begin{center}
    \begin{tabular}{c}
      \begin{minipage}{0.45\hsize}
        \begin{center}
          \includegraphics[width=.95\linewidth]{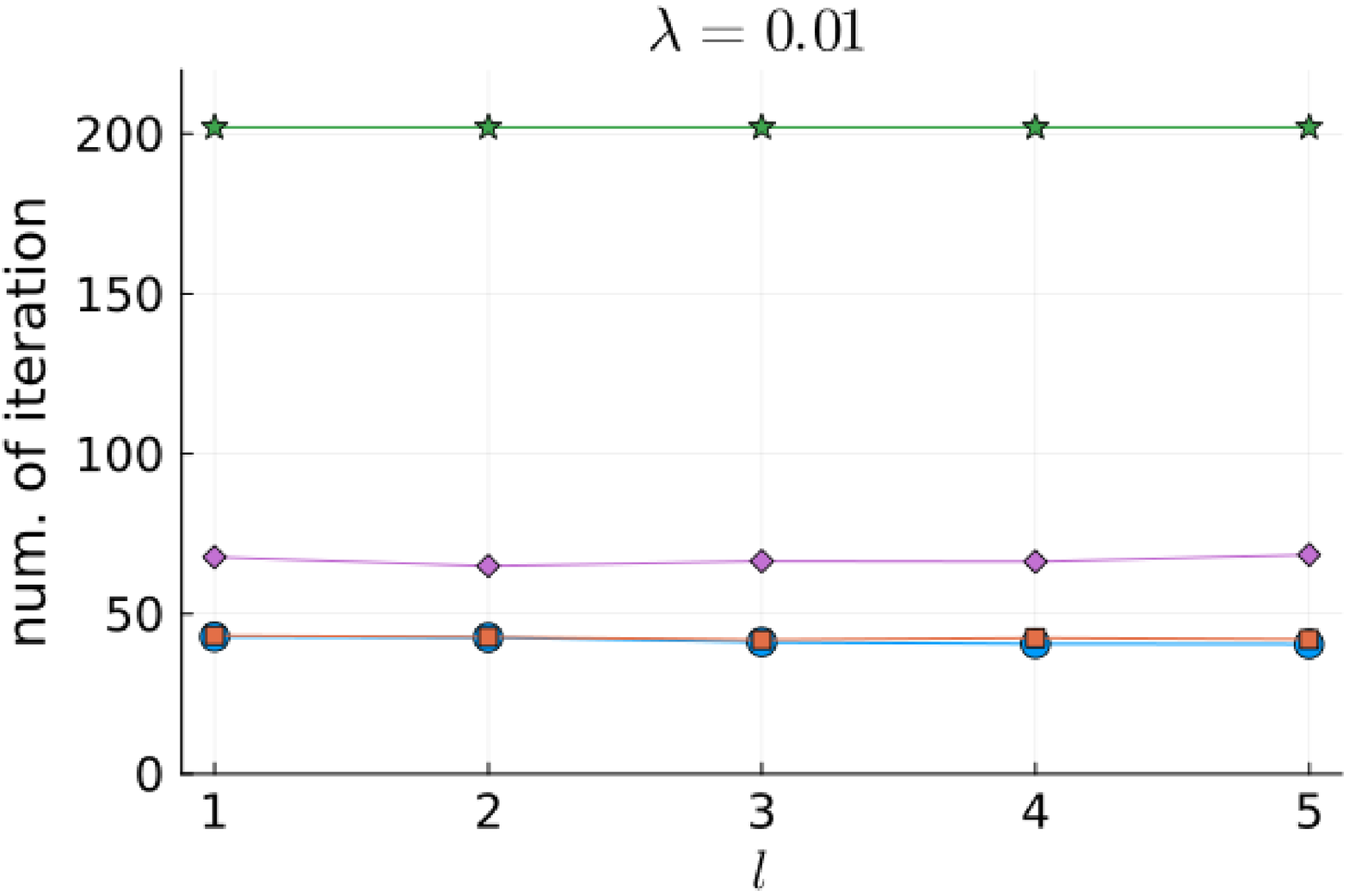}
                  \end{center}
      \end{minipage}
      \begin{minipage}{0.45\hsize}
        \begin{center}
          \includegraphics[width=.95\linewidth]{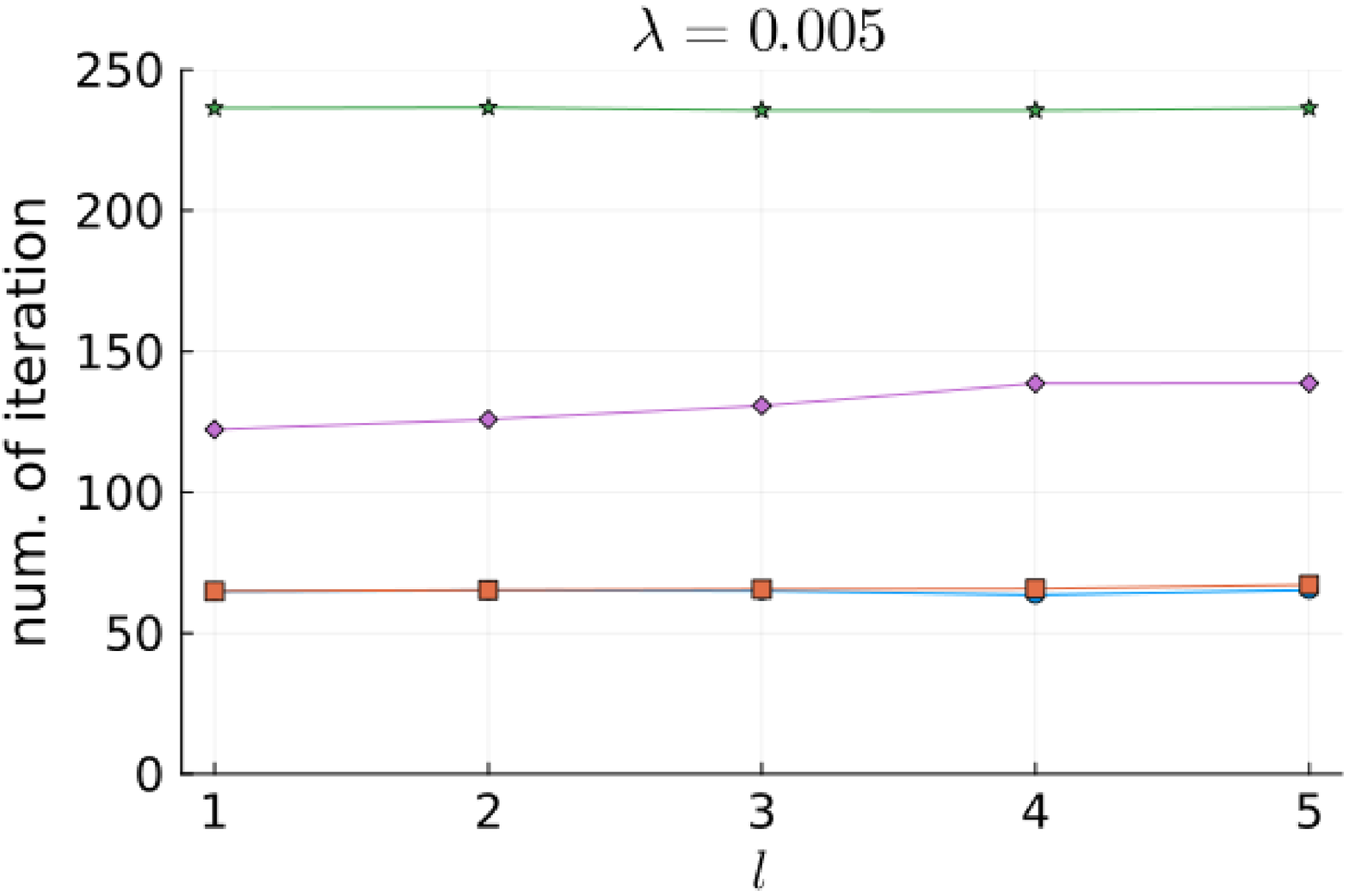}
                  \end{center}
      \end{minipage}
      \\
      \begin{minipage}{0.45\hsize}
        \begin{center}
          \includegraphics[width=.95\linewidth]{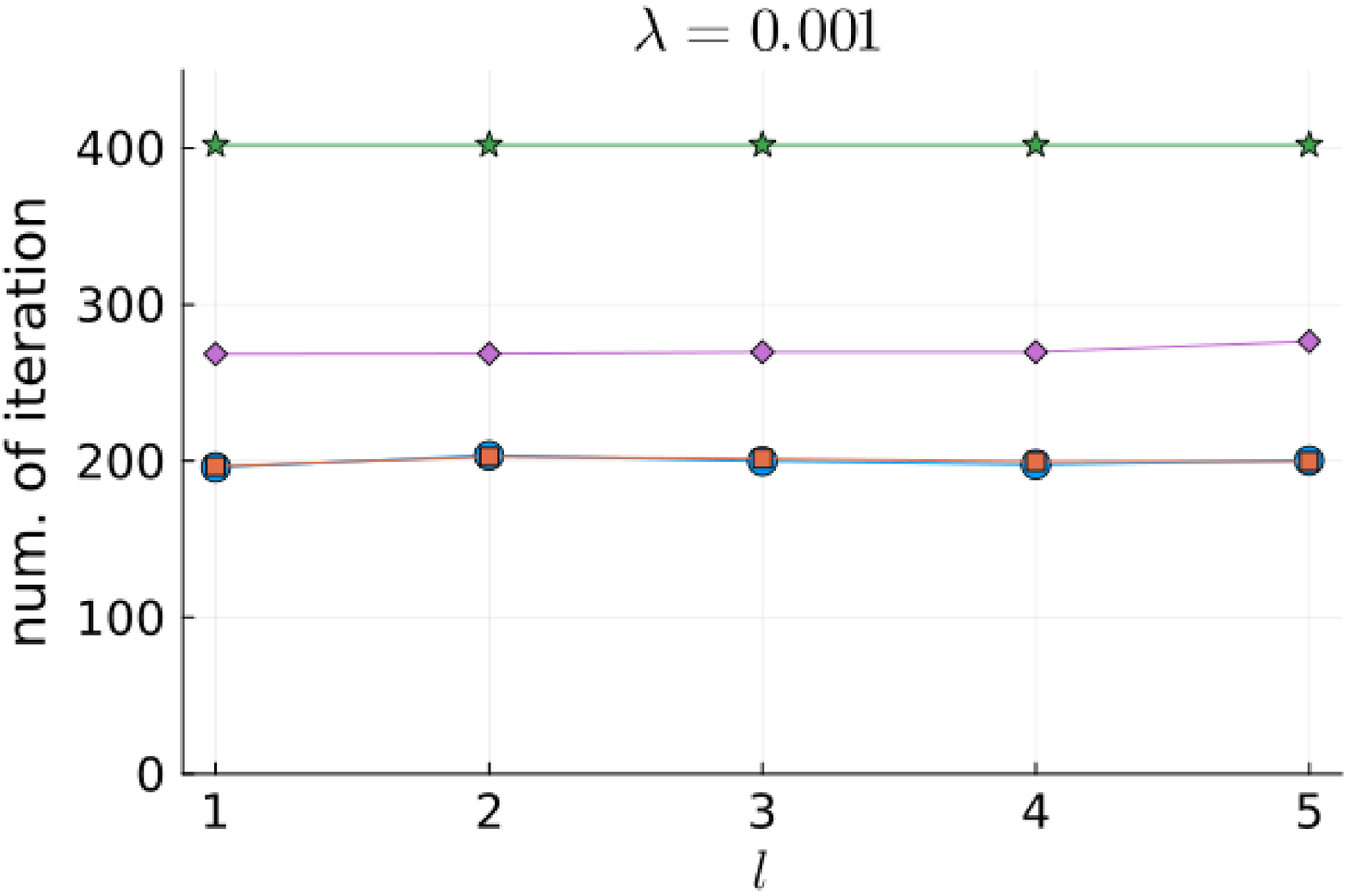}
                  \end{center}
      \end{minipage}
      \begin{minipage}{0.45\hsize}
        \begin{center}
          \includegraphics[width=.95\linewidth]{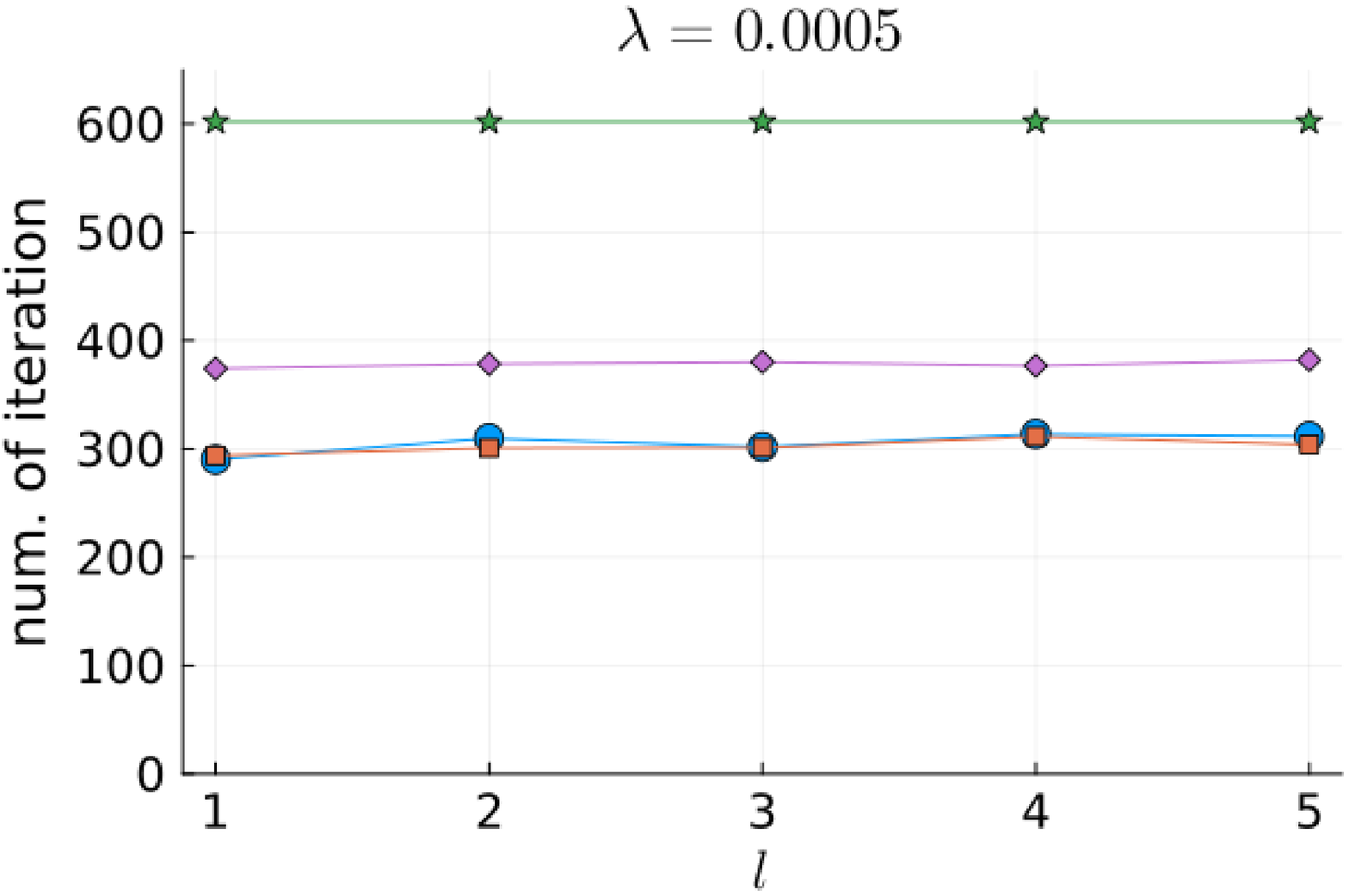}
                  \end{center}
      \end{minipage}
    \end{tabular}
    \caption{Average number of iterations to solve \eqref{LSP}}
    \label{fig:iter-LSP}
  \end{center}
\end{figure}

\section{Concluding remarks}\label{sec:con}
We proposed an inexact proximal DC Newton-type method (Algorithm~\ref{alg:proximal-DC-Newton}) and showed its global convergence properties. 
We established concrete choices for the memoryless quasi-Newton matrices \eqref{SBroyden_B}  for the scaled proximal mappings. 
Moreover, we adopted the semi-smooth Newton method (Algorithm~\ref{alg:semismoothNewton}) in the computing scaled proximal mappings. 
In our numerical experiments, the proposed algorithm outperformed existing methods for two classes of DC regularized least squares problems.

\begin{appendices}
\section{Proof of Lemma \ref{lem_line}}\label{Appendix:lemma1}
\begin{proof}
It follows from  $\eta\in(0,1]$, $x_k+\eta d_k = \eta x_k^+ + (1-\eta)x_k$ and the convexity of $h_1$ 
that
\[
h_1(x_k+\eta d_k) \leq \eta h_1(x_k^+) +(1-\eta)h_1(x_k).
\]
On the other hand, $\xi_k\in\partial h_2(x_k)$ implies
\[
 h_2(x_k)+\eta \xi_k^Td_k \leq 
h_2(x_k+\eta d_k).
\]
From the inequalities and Assumption~\ref{ass:Lip}, 
we obtain
\begin{align*}
f(x_k + \eta d_k)-f(x_k) &= g(x_k + \eta d_k)-g(x_k)  \\ 
 & \qquad + h_1(x_k + \eta d_k) -h_1(x_k)-h_2(x_k + \eta d_k)+h_2(x_k) \\
	&\leq (\nabla g(x_k)-\xi_k)^T(\eta d_k) + \frac{L}{2}\|\eta d_k\|^2 + \eta \left(h_1(x_k^+) - h_1(x_k)\right)\\
	&=\eta \left((\nabla g(x_k)-\xi_k)^Td_k+h_1(x_k^+)-h_1(x_k)\right)+\frac{\eta^2L}{2}\|d_k\|^2 .
\end{align*}
Therefore, \eqref{lem_alpha} holds.

Since it follows from \eqref{search} and \eqref{sub_opt} that
\begin{equation*}\label{subgra_h+}
r_k - \nabla g(x_k) + \xi_k -B_kd_k \in \partial h_1(x_k^+),
\end{equation*}
we obtain 
\begin{equation*}\label{lem2_3}
 h_1(x_k^+) + (r_k - \nabla g(x_k) +\xi_k- B_k d_k)^T(-d_k) \leq h_1(x_k).
\end{equation*}
Hence, we have
\begin{equation}\label{lem2_3_2}
(\nabla g(x_k)-\xi_k)^Td_k + h_1(x_k^+) - h_1(x_k)  \leq r_k^Td_k-\|d_k\|_{B_k}^2.
\end{equation}
Using \eqref{inexact} and the Cauchy-Schwarz inequality, we get
\begin{equation}\label{lem2_2}
r_k^Td_k-\|d_k\|^2_{B_k}\leq\|r_k\|_{H_k}\|d_k\|_{B_k} - \|d_k\|_{B_k}^2\leq -\bar\theta \|d_k\|_{B_k}^2.
\end{equation}
Combining \eqref{lem2_3_2} with \eqref{lem2_2}, we obtain \eqref{decrease2}, completing the proof.
\end{proof}

\section{Proof of Lemma \ref{alpha_lemma}}\label{Appendix:lemma2}
\begin{proof}
For any $0<\eta \leq \frac{2m}{L}\bar\theta(1-\delta)$, we have  from (\ref{uniformly_mat}) and \eqref{decrease2}, 
\begin{align*}
 \frac{L\eta}{2}\|d_k\|^2 &\leq m\bar\theta(1-\delta)\|d_k\|^2\\
 &\leq(1-\delta)\bar\theta \|d_k\|_{B_k}^2\\
  &\leq -(1-\delta)((\nabla g(x_k)-\xi_k)^Td_k + h_1(x_k^+) - h_1(x_k)).
\end{align*}
Hence, it follows from (\ref{lem_alpha}) that
\[f(x_k + \eta d_k)-f(x_k) \leq \eta \delta((\nabla g(x_k)-\xi_k)^Td_k + h_1(x_k^+) - h_1(x_k)).\]
This means that the line search condition \eqref{linecond} is satisfied for all
\[0<\eta \leq \min\left\{1,\frac{2m}{L}\bar\theta(1-\delta)\right\}.\]
Therefore, since we use the backtracking line search with $\beta_k\in(0,1)$, 
\[\beta_k\min\left\{1,\frac{2m}{L}\bar\theta(1-\delta)\right\} \leq \eta_k\leq1\]
holds. 
It follows from the above and $\beta_{min}\leq\beta_k$ that we have \eqref{alpha_lem_back}. 
Hence, this lemma is proved. 
\end{proof}

\section{Proof of Theorem \ref{thm:prox}}\label{Appendix:prox}
To prove Theorem \ref{thm:prox}, we introduce the following theorem \cite[Theorem 3.4]{SIOPT_Becker}.
\begin{theorem}\label{thm:Beker}
Let $V=D\pm \sum_{i=1}^r u_iu_i^T \in\mathbb{R}^{n\times n}$ be symmetric positive definite, where
$D\in\mathbb{R}^{n\times n}$ is symmetric positive definite and $u_i\in\mathbb{R}^n$.
Let $U=(u_1,...,u_r)$. If $r\leq n$, $U$ is full rank and $h_1$ is proper lsc convex, then 
\begin{equation*}\label{Becker_prox1}
{\rm Prox}_{h_1}^V(\bar{x}) = {\rm Prox}_{h_1}^D(\bar{x} \mp D^{-1}U\alpha^\ast),
\end{equation*}
where the mapping $\mathcal{L}:\mathbb{R}^r\to\mathbb{R}^r$ is defined by
\begin{equation*}\label{Berker_prox2}
\mathcal{L}(\alpha) = U^T(\bar{x} - {\rm Prox}_{h_1}^D(\bar{x} \mp D^{-1}U\alpha)) + \alpha
\end{equation*}
and $\alpha^\ast\in\mathbb{R}^r$ is the unique root of $\mathcal{L}(\alpha)=0$.
\end{theorem}

By using this theorem, we can prove Theorem \ref{thm:prox}. 
\begin{proof} ({\bf Proof of Theorem \ref{thm:prox}})
Let $P =\tau I + u_1u_1^T$, $B= P - u_2u_2^T$. Then, from Theorem \ref{thm:Beker} with $V=B$ and $D=P$, we have
\begin{equation*}\label{Becker_prox3}
{\rm Prox}_{h_1}^{B}(\bar{x}) = {\rm Prox}_{h_1}^{P} (\bar{x} + \alpha_2^\ast P^{-1}u_2),
\end{equation*}
where the mapping $\mathcal{L}_2:\mathbb{R}\to\mathbb{R}$ is defined by
\begin{equation*}\label{Becker_prox4}
\mathcal{L}_2(\alpha_2) = u_2^T(\bar{x} - {\rm Prox}_{h_1}^{P} (\bar{x}  + \alpha_2P^{-1}u_2)) + \alpha_2
\end{equation*}
and $\alpha_2^\ast\in\mathbb{R}$ is the root of $\mathcal{L}_2(\alpha_2)=0$. 
We next consider ${\rm Prox}_{h_1}^{P} (\bar{x} + \alpha_2^\ast P^{-1}u_2)$.
Applying Theorem \ref{thm:Beker} with $D = \tau I $ and $V = P$, we have
\begin{equation*}\label{Becker_prox5}
{\rm Prox}_{h_1}^{P}(\bar{x} + \alpha_2^\ast P^{-1}u_2) = {\rm Prox}_{h_1}^{\tau I} (\bar{x} + \alpha_2^\ast P^{-1}u_2 - \frac{\alpha_1^\ast}{\tau} u_1)
\end{equation*}
where   the mapping $\mathcal{L}_1:\mathbb{R}\to\mathbb{R}$ is defined by 
\begin{equation*}\label{Becker_prox6}
\mathcal{L}_1(\alpha_1) = u_1^T(\bar{x} + \alpha_2^\ast P^{-1}u_2 - {\rm Prox}_{h_1}^{\tau I} (\bar{x} + \alpha_2^\ast P^{-1}u_2 - \frac{\alpha_1}{\tau}u_1)) + \alpha_1
\end{equation*}
and $\alpha_1^\ast\in\mathbb{R}$ is the root of $\mathcal{L}_1(\alpha_1)=0$.
We now note that 
\[
{\rm Prox}_{h_1}^{\tau I}(\cdot) = \argmin_{x\in\mathbb{R}^n}~ h_1(x)+\frac{1}{2}\|x-\cdot\|_{\tau I}^2=
\argmin_{x\in\mathbb{R}^n}~ \frac{1}{\tau}h_1(x)+\frac{1}{2}\|x-\cdot\|^2={\rm Prox}_{\frac{1}{\tau}h_1}(\cdot).\]
Summarizing the above relations, we have \eqref{prox_Bk}.

We next aim to show the existence and the uniqueness of the solution $\alpha^\ast$.
The existence is immediately guaranteed by Theorem~\ref{thm:Beker}. 
To show uniqueness,  we choose any two solutions of $\mathcal{L}(\alpha)=0$, say $\hat\alpha=(\hat\alpha_1,\hat\alpha_2)^T,\ \bar\alpha=(\bar\alpha_1,\bar\alpha_2)^T\in\mathbb{R}^2$. 
Then, it follows from $\mathcal{L}(\hat\alpha)=\mathcal{L}(\bar\alpha)$ that 
\[
\left\{
\begin{array}{l}
u_1^T(\hat\alpha_2(\tau I+u_1u_1^T)^{-1}u_2-{\rm Prox}_{\frac{1}{\tau}h_1}(\zeta(\hat\alpha)))+\hat\alpha_1\\
\qquad =u_1^T(\bar\alpha_2(\tau I+u_1u_1^T)^{-1}u_2-{\rm Prox}_{\frac{1}{\tau}h_1}(\zeta(\bar\alpha)))+\bar\alpha_1,\\[6pt]
-u_2^T{\rm Prox}_{\frac{1}{\tau}h_1}(\zeta(\hat\alpha))+\hat\alpha_2
=-u_2^T{\rm Prox}_{\frac{1}{\tau}h_1}(\zeta(\bar\alpha))+\bar\alpha_2. 
\end{array}
\right.
\]
Thus, the relations ${\rm Prox}_{\frac{1}{\tau}h_1}(\zeta(\bar\alpha))={\rm Prox}_{h_1}^{B}(\bar{x}) ={\rm Prox}_{\frac{1}{\tau}h_1}(\zeta(\hat\alpha))$ and the second equality yield 
$\hat\alpha_2=\bar\alpha_2$. 
Further, the first equality implies $\hat\alpha_1=\bar\alpha_1$. 
Therefore, we have $\hat\alpha=\bar\alpha$, which implies that the solution of $\mathcal{L}(\alpha)=0$ is unique, completing the proof.
\end{proof}

\section{Proof of Proposition \ref{prop:res}}\label{Appendix:prox2}
\begin{proof}
For simplicity, we set $\hat{x}={\rm Prox}_{\frac{1}{\tau}h_1}(\zeta(\alpha))$. 
It follows from \eqref{MLBFGS_B2}, \eqref{def:L} and $
u_1u_1^T(\tau I+u_1u_1^T)^{-1}=I-\tau(\tau I+u_1u_1^T)^{-1}
$ that
\begin{align}
U\mathcal{L}(\alpha)&=
-u_1u_1^T\bar x - \alpha_2u_1u_1^T(\tau I+u_1u_1^T)^{-1} u_2+u_1u_1^T \hat{x} - \alpha_1 u_1 
+ u_2u_2^T\bar x  -u_2u_2^T\hat{x} + \alpha_2 u_2 
\nonumber\\
&=(-u_1u_1^T+u_2u_2^T)(\bar{x}-\hat{x}) - \alpha_2u_1u_1^T(\tau I+u_1u_1^T)^{-1} u_2 - \alpha_1 u_1 + \alpha_2 u_2
\nonumber\\
&=(\tau I-B)(\bar{x}-\hat{x}) - \alpha_2u_2 +  \alpha_2\tau(\tau I+u_1u_1^T)^{-1} u_2 - \alpha_1 u_1 + \alpha_2 u_2
\nonumber\\
&=B(\hat{x}-\bar x) + \tau(\bar{x}-\hat{x}) - \alpha_1 u_1 + \tau\alpha_2(\tau I+u_1u_1^T)^{-1} u_2.
\label{UL}
\end{align}
On the other hand, $\hat{x}={\rm Prox}_{\frac{1}{\tau}h_1}(\zeta(\alpha))$ implies
\[
\tau(\zeta(\alpha)-\hat{x})\in \partial h_1(\hat{x}).
\]
Therefore, it follows from \eqref{def:zeta}, \eqref{UL}, $\bar{x}=x-H(\nabla g(x)- \xi)$, and $BH=I$ that
\begin{align*}
U\mathcal{L}(\alpha)&=B(\hat{x}-\bar x) +\tau\left( \bar x -\hat{x}- \frac{
\alpha_1
}{\tau} u_1 + \alpha_2(\tau I+u_1u_1^T)^{-1} u_2 \right)\\
&=B(\hat{x}-\bar x)+\tau( \zeta(\alpha)-\hat{x})\\
&=\nabla g(x)- \xi + B(\hat{x}-x)+\tau(\zeta(\alpha)-\hat{x})\\
&\in \nabla g(x) - \xi + B(\hat{x}-x)+\partial h_1(\hat{x}).
\end{align*}
This completes the proof. 
\end{proof}
\section{Proof of Theorem \ref{thm_global_semismooth}}\label{Appendix:global_semismooth}
To prove Theorem \ref{thm_global_semismooth}, we first give the following lemma.
\begin{lemma}\label{lem_inner_linesearch}
Assume that ${\rm Prox}_{\frac{1}{\tau}h_1}$ is B-differentiable. 
Let $\bar\alpha\in\mathbb{R}^2$ be a point such that 
$\mathcal{L}(\bar\alpha)\ne 0$ and 
any element of $\partial^C \mathcal{L}(\bar\alpha)$ is nonsingular. 
Then, there exist a positive constant $\bar t$ and a compact neighborhood $\mathcal{N}(\bar\alpha)$ of $\bar\alpha$ 
such that the following statements hold for any $\alpha\in\mathcal{N}(\bar\alpha)$:
\begin{enumerate}
\item[(a)] 
$\mathcal{L}(\alpha)\ne 0$ and 
any element of $\partial^C \mathcal{L}(\alpha)$ is nonsingular.
\item[(b)] 
For $p=-V^{-T}\mathcal{L}(\alpha)$ and $V\in\partial^C \mathcal{L}(\alpha)$ satisfying 
\begin{align}
\Psi^\prime (\alpha;p)\le (V\mathcal{L}(\alpha))^Tp,
\label{cond_add}
\end{align}
the inequality
\begin{align}
\Psi(\alpha+t p)\le (1-2\sigma t)\Psi(\alpha)
\label{armijo_tmp}
\end{align}
holds for any $t\in(0,\bar t]$.
\end{enumerate}
\end{lemma}
\begin{proof} 
Since ${\rm Prox}_{\frac{1}{\tau}h_1}$ is local Lipschitz continuous, $\mathcal{L}$ is also local Lipschitz continuous, and so $\partial^C \mathcal{L}(\alpha)$ is compact for any $\alpha$. 
Since any element of $\partial^C \mathcal{L}(\bar\alpha)$ is nonsingular, 
there exists a compact neighborhood 
$\mathcal{T}(\bar\alpha) \supset \partial^C \mathcal{L}(\bar\alpha)$ such that any element of $\mathcal{T}(\bar\alpha)$ is nonsingular. 
Because $\partial^C \mathcal{L}$ is upper semi-continuous and $\partial^C \mathcal{L}(\alpha)$ is compact for any $\alpha$, 
we can choose $\mathcal{T}(\bar\alpha) \supset \partial^C \mathcal{L}(\bar\alpha)$  and a compact neighborhood $\mathcal{N}(\bar\alpha)$ of $\bar\alpha$ such that  
$\mathcal{L}(\alpha)\ne 0$ and $\partial^C \mathcal{L}(\alpha)\subset \mathcal{T}(\bar\alpha)$ hold for any $\alpha\in\mathcal{N}(\bar\alpha)$. 
Thus, (a) is satisfied. 

Next, we show (b). 
Since ${\rm Prox}_{\frac{1}{\tau}h_1}$ is local Lipschitz continuous and directionally differentiable, $\mathcal{L}$ is B-differentiable~{\cite[Definition 3.1.2]{facchinei2003finiteV1}}.  
Thus, it follows from \eqref{cond_add} and \cite[Proposition 3.1.3]{facchinei2003finiteV1} that the following relations hold for any $t>0$: 
\begin{align}
\Psi(\alpha+t p)&=\Psi(\alpha)+\Psi^\prime (\alpha;tp)+o(\|t p\|)\nonumber\\
&=\Psi(\alpha)+t\Psi^\prime (\alpha;p)+o(\|t p\|)\nonumber\\
&\le\Psi(\alpha)+t(V\mathcal{L}(\alpha))^T p+o(\|t p\|)\nonumber\\
&=\Psi(\alpha)-t \|\mathcal{L}(\alpha)\|^2+o(\|t p\|)\nonumber\\
&=(1-2t)\Psi(\alpha)+o(\|t p\|). \label{eval:Psi}
\end{align}
From the above arguments, for any $\alpha\in\mathcal{N}(\bar{\alpha})$, it holds that $\partial^C \mathcal{L}(\alpha)\subset \mathcal{T}(\bar\alpha)$ and $\mathcal{T}(\bar\alpha)$ is compact. Hence, $p=-V^{-T}\mathcal{L}(\alpha)$ is bounded. 
In addition, since $\mathcal{N}(\bar{\alpha})$ is compact and $\mathcal{L}(\alpha)\ne 0$ for any $\alpha\in \mathcal{N}(\bar{\alpha})$, 
there exists a positive constant $\tilde \Psi$ such that $\tilde\Psi\le\Psi(\alpha)$ for any $\alpha\in \mathcal{N}(\bar{\alpha})$. 
Therefore, it follows from $\sigma\in(0,1/2)$ and \eqref{eval:Psi} that 
\[
\Psi(\alpha+t p)\le(1-2\sigma t)\Psi(\alpha) -2t(1-\sigma)\tilde \Psi +o(t). 
\]
Thus, there exists a positive constant $\bar{t}$ such that 
\eqref{armijo_tmp} holds for any $t\in(0,\bar t]$. 
\end{proof}
From Lemma~\ref{lem_inner_linesearch}, we immediately have the following property. 
\begin{remark}\label{welldef_inner_linesearch}
Consider Algorithm~\ref{alg:semismoothNewton}. 
If any element of $\partial^C \mathcal{L}(\alpha_j)$ is nonsingular and  
  \eqref{eq:descent_direction} holds, 
then the line search condition \eqref{inner_armijo} is achieved for some finite number $l$. 
\end{remark}

By using Lemma~\ref{lem_inner_linesearch}, we prove Theorem~\ref{thm_global_semismooth}. 
\begin{proof} ({\bf Proof of Theorem~\ref{thm_global_semismooth}})
If $\mathcal{L}(\alpha_j)=0$ for some $j\ge 0$, we have the desired result. 
Thus, we consider the case where $\mathcal{L}(\alpha_j)\neq 0$ for all $j\ge 0$. 
It follows from Remark~\ref{welldef_inner_linesearch} and the line search condition \eqref{inner_armijo} that $\{\Psi(\alpha_j)\}$ is a nonincreasing sequence. Hence, $\{\alpha_j\}\subset\mathcal{S}_0$ holds. 
Since the level set $\mathcal{S}_0$ is compact, $\{\alpha_j\}$ has at least one accumulation point. 

We show the theorem by contradiction. 
Assume that there exists an accumulation point $\widehat{\alpha}$ 
such that $\mathcal{L}(\widehat\alpha)\ne 0$ (namely, $\Psi (\widehat\alpha)>0$),  
and consider a subsequence $\{\alpha_{j_i}\}$ 
such that $\{\alpha_{j_i}\}\to\widehat{\alpha}\ (i\to\infty)$. 
For sufficiently large $i$,  the relation $\{\alpha_{j_i}\}\subset \mathcal{N}(\widehat{\alpha})$ holds, 
where $\mathcal{N}(\widehat{\alpha})$ is the neighborhood appearing in Lemma~\ref{lem_inner_linesearch} with $\bar\alpha=\widehat{\alpha}$. 
Let $\hat{l}$ be the smallest nonnegative integer such that $\rho^{\hat{l}}\le\bar t$, where $\bar t$ is the positive constant appearing in Lemma~\ref{lem_inner_linesearch}. 
Then, it follows from \eqref{armijo_tmp} that 
\begin{align*}
\Psi\left(\alpha_{j_i}+\rho^{\hat{l}} p_{j_i}\right)\le \left(1-2\sigma \rho^{\hat{l}}\right)\Psi(\alpha_{j_i})
\end{align*}
holds for sufficiently large $i$. 
From the backtracking rule of the algorithm, $\rho^{\hat{l}}\le t_{j_i}$ is satisfied. 
Hence, taking into account $j_i+1\le j_{i+1}$, we have 
\begin{align*}
\Psi(\alpha_{j_{i+1}})\le\Psi(\alpha_{j_i+1})=\Psi(\alpha_{j_i}+t_{j_i} p_{j_i})
\le (1-2\sigma t_{j_i})\Psi(\alpha_{j_i})\le \left(1-2\sigma \rho^{\hat{l}}\right)\Psi(\alpha_{j_i}). 
\end{align*}
Since $1-2\sigma \rho^{\hat{l}}\in(0,1)$ is a constant independent of $i$, 
we obtain 
\[
\Psi(\widehat{\alpha})=\lim_{i\to\infty}\Psi(\alpha_{j_i})=0. 
\]
Since this contradicts the assumption $\mathcal{L}(\widehat\alpha)\ne 0$, 
any accumulation point of $\{\alpha_j\}$ is a solution of \eqref{solve:L}. 
Moreover, from Theorem~\ref{thm:prox}, problem \eqref{solve:L} has a unique solution. Hence, the proof is complete. 
\end{proof}


\section{Proof of Theorem \ref{thm:conv}}\label{Appendix:local_semismooth}
\begin{proof}
It follows from Theorem~\ref{thm_global_semismooth}, the sequence $\{\alpha_j\}$ converges to the unique solution $\alpha^\ast$. 
{
In the same way as the proof of Lemma~\ref{lem_inner_linesearch}(a),  
we can show that there exists a compact neighborhood $\mathcal{N}^\prime(\alpha^\ast)$ such that 
any element of $\partial^C \mathcal{L}(\alpha)$ is nonsingular for any $\alpha\in\mathcal{N}^\prime(\alpha^\ast)$.
Since $\mathcal{N}(\alpha^\ast)$ is a compact set,  
$\partial^C \mathcal{L}$ is upper semi-continuous, and $\alpha_j\in\mathcal{N}(\alpha^\ast)$ for sufficiently large $j$, 
there exists a positive constant $\widehat{c}_1$ such that
\[
\|V_j^{-1}\|\le \widehat{c}_1 \quad \mbox{for }\forall V_j\in \partial^C \mathcal{L}(\alpha_j)
\]
holds. 
}
Therefore, the (strongly) semi-smoothness yields 
\begin{align}
\|\alpha_j+p_j-\alpha^\ast\|&=\|\alpha_j-V_j^{-T}\mathcal{L}(\alpha_j)-\alpha^\ast\|\nonumber\\
&\le \widehat{c}_1\|V_j^T(\alpha_j-\alpha^\ast)-\mathcal{L}(\alpha_j)+\mathcal{L}(\alpha^\ast)\|=o(\|\alpha_j-\alpha^\ast\|)\label{super_linear}\\
&(=O(\|\alpha_j-\alpha^\ast\|^2\ \ \mbox{for the strongly semi-smooth case}).\nonumber
\end{align}
On the other hand, from the local Lipschitz continuity of $\mathcal{L}$ and \cite[Theorem 3.1]{MOR_Qi1993}, 
there exist positive constants $\widehat{c}_2,\ \widehat{c}_3$ satisfying 
\[
\widehat{c}_2\|\alpha_j-\alpha^\ast\|\le\|\mathcal{L}(\alpha_j)-\mathcal{L}(\alpha^\ast)\|
\le \widehat{c}_3 \|\alpha_j-\alpha^\ast\|. 
\]
Therefore, by \eqref{super_linear}, we have 
\begin{align*}
\Psi(\alpha_j+p_j)&=\frac12 \|\mathcal{L}(\alpha_j+p_j)-\mathcal{L}(\alpha^\ast)\|^2\\
&=O(\|\alpha_j+p_j-\alpha^\ast\|^2)=o(\|\alpha_j-\alpha^\ast\|^2)=o(\|\mathcal{L}(\alpha_j)\|^2)=o(\Psi(\alpha_j)),
\end{align*}
which implies that the line search condition \eqref{inner_armijo} holds with $l=0$, namely, $t_j=1$. 
Thus, using \eqref{super_linear}, we obtain 
\begin{align*}
\|\alpha_{j+1}-\alpha^\ast\|&=o(\|\alpha_j-\alpha^\ast\|)\\
&(=O(\|\alpha_j-\alpha^\ast\|^2\ \ \mbox{for the strongly semi-smooth case}),
\end{align*}
and hence the proof is complete. 
\end{proof}

\section{Proof of Proposition~\ref{prop2pprime}}
\label{sec:A6}
\begin{proof}
The definition \eqref{def:u} yields 
\[
u_2^Tu_2=\tau_k,\quad
u_1^Tu_1=\frac{\gamma_k z_{k-1}^Tz_{k-1}}{s_{k-1}^Tz_{k-1}}, \quad
u_1^Tu_2=\frac{\sqrt{\tau_k\gamma_ks_{k-1}^Tz_{k-1}}}{\|s_{k-1}\|}. 
\]
It follows from $s_{k-1}^Tz_{k-1}>0$ and the Cauchy--Schwarz inequality that 
\[
\frac{s_{k-1}^Tz_{k-1}}{s_{k-1}^Ts_{k-1}}\le \frac{z_{k-1}^Tz_{k-1}}{s_{k-1}^Tz_{k-1}}.
\] 
Therefore, using \eqref{Lip_Ass}, \eqref{def:sz}, \eqref{sz>m}, and  \eqref{<gamma<}, we have
\[
\underline{\tau}\le u_2^Tu_2\le \bar\tau,\quad 
\underline{\gamma}\underline{\nu}
\le\frac{\gamma_k s_{k-1}^Tz_{k-1}}{s_{k-1}^Ts_{k-1}}
\le u_1^Tu_1\le\frac{\bar\gamma(\bar{\nu}+L)^2}{\underline{\nu}},
\]
and 
\[
\sqrt{\underline{\tau}\underline{\gamma}\underline{\nu}}\le u_1^Tu_2
\le\sqrt{\bar{\tau}\bar{\gamma}(\bar{\nu}+L)}. 
\]
 From $(\tau_k I+u_1u_1^T)^{-1}=\frac{1}{\tau_k}I-\frac{u_1u_1^T}{\tau_k^2+\tau_k\|u_1\|^2}$, we get 
\[
u_2^T(\tau_k I+u_1u_1^T)^{-1}u_2=\frac{1}{\tau_k} u_2^Tu_2-\frac{(u_1^Tu_2)^2}{\tau_k^2+\tau_k\|u_1\|^2} 
=1-\frac{(u_1^Tu_2)^2}{\tau_k^2+\tau_k\|u_1\|^2}, 
\]
which implies that 
\[
\frac{\underline{\tau}\underline{\gamma}\underline{\nu}}{\bar{\tau}^2+\bar{\tau}\frac{\bar\gamma(\bar{\nu}+L)^2}{\underline{\nu}}}
\le
1-u_2^T(\tau_k I+u_1u_1^T)^{-1}u_2
\le \frac{(u_1^Tu_2)^2}{\tau_k^2}
\le \frac{\bar{\tau}\bar{\gamma}(\bar{\nu}+L)}{\underline{\tau}^2}. 
\]
By letting $v=\tau_k(\zeta(\alpha) -   {\rm Prox}_{\frac{1}{\tau_k}h_1}(\zeta(\alpha)))\in\partial h_1(\zeta(\alpha))$, 
it follows from \eqref{def:zeta} and \eqref{def:L} that 
\begin{align}
\mathcal{L}(\alpha) =
\begin{pmatrix}
\frac{1}{\tau_k}u_1^Tv+(1+\frac{1}{\tau_k}u_1^Tu_1)\alpha_1\\[6pt]
\frac{1}{\tau_k}u_2^Tv+\frac{1}{\tau_k}u_1^Tu_2 \alpha_1+(1-u_2^T(\tau_k I+u_1u_1^T)^{-1}u_2)\alpha_2 
\end{pmatrix}. 
\label{rewite:L}
\end{align}
On the other hand, from the assumption \eqref{ass:bound_cdiff} and the above evaluations, 
the following relations hold: 
\[
\lvert u_1^Tv \rvert \le\bar{c} \sqrt{\frac{\bar\gamma(\bar{\nu}+L)^2}{\underline{\nu}}},\quad \lvert u_2^Tv \rvert \le \bar{c}\sqrt{\bar{\tau}}. 
\]
Therefore, it follows from the above evaluations, \eqref{<gamma<}, 
and \eqref{rewite:L} that there exist  positive constants $\widehat{c}_4,\widehat{c}_5$, and $\widehat{c}_6$ satisfying 
\[\widehat{c}_4\alpha_1^2+(\widehat{c}_5\alpha_1+\widehat{c}_6\alpha_2)^2\le\frac12 \|\mathcal{L}(\alpha)\|^2=\Psi(\alpha)\]
when $\|\alpha\|$ is sufficiently large. 
Therefore, the proof is complete. 
\end{proof}

\section{Choice for $V_j$}\label{sec:Adirectional}
\begin{proposition}\label{prop:directional}
Suppose that $h_1(x)=\lambda\|x\|_1$ $(\lambda>0)$. 
Let $\zeta$ and $\mathcal{L}$ be given in \eqref{def:zeta} and \eqref{def:L}, and  let
\begin{equation}\label{eq:jac}
V_j = \begin{pmatrix}
1+\frac{1}{\tau}u_1^TWu_1 &\frac{1}{\tau}u_2^TWu_1 \\[5pt]
(u_1 - Wu_1)^T(\tau I+u_1u_1^T)^{-1}u_2 & ~1-u_2^TW(\tau I+u_1u_1^T)^{-1}u_2
\end{pmatrix},
\end{equation}
where
\begin{equation*}
W
= \begin{pmatrix}
  w_1             &\\
    &           \ddots        &             \\
    &        &      w_n
\end{pmatrix}
\quad\text{and}\quad
w_i=
\left\{ 
\begin{array}{ll}
1~~ & \text{if}~ \lvert(\zeta(\alpha_j))_i\rvert >  \frac{\lambda}{\tau} 
, \\[5pt]
0~~ &\text{otherwise},
\end{array} \right.
\end{equation*}
for $i=1,\dots,n$. 
Then, $V_j\in \partial^C \mathcal{L}(\alpha_j)$ holds. 
\end{proposition}
\begin{proof}
For simplicity, we omit the subscript $j$  and 
set 
\[\bar{u}_1 = \frac{1}{\tau}u_1\quad\text{and}\quad \bar{u}_2 = (\tau I+u_1u_1^T)^{-1}u_2.\]
Then, we can rewrite $\zeta(\alpha)$ and $\mathcal{L}(\alpha) $ as 
\begin{equation*}
\zeta(\alpha) = \bar{x} - \alpha_1 \bar{u}_1 + \alpha_2 \bar{u}_2
\end{equation*}
and
\begin{align*}
\mathcal{L}(\alpha) 
=
\begin{pmatrix}
 \alpha_1 + 
u_1^T\bar{x}  +  \alpha_2(u_1^T\bar{u}_2) -  u_1^T {\rm Prox}_{\frac{1}{\tau}h_1}(\zeta(\alpha))\\
 \alpha_2 + u_2^T\bar{x} -  u_2^T {\rm Prox}_{\frac{1}{\tau}h_1}(\zeta(\alpha))
\end{pmatrix},
\end{align*}
respectively. When $h_1(x)=\lambda\|x\|_1$ $(\lambda>0)$ , the proximal mapping is given by
\[
\left({\rm Prox}_{\frac{1}{\tau}h_1}(\zeta(\alpha))\right)_i 
=
\left\{ 
\begin{array}{ll}
(\zeta(\alpha))_i- \frac{\lambda}{\tau}  & \text{if}~(\zeta(\alpha))_i \geq \frac{\lambda}{\tau}, \\
0 &  \text{if}~\lvert (\zeta(\alpha))_i  \rvert <  \frac{\lambda}{\tau},\\
(\zeta(\alpha))_i + \frac{\lambda}{\tau} & \text{if}~ (\zeta(\alpha))_i \leq -\frac{\lambda}{\tau}.
\end{array} \right.
\] 
We now consider $\mathcal{D} = \{\alpha \vert \mathcal{L}(\alpha) \text{ is differenciable}\}$. 
For $\forall \alpha\in \mathcal{D}$, we have
\begin{equation*}
\nabla \mathcal{L}(\alpha) = 
\begin{pmatrix}\displaystyle
1 + \sum_{i=1}^n (u_1)_i (\bar{u}_1)_i \bar\omega_{i} 
& \displaystyle  \sum_{i=1}^n(\bar{u}_1)_i (u_2)_i \bar\omega_{i} \\
\displaystyle u_1^T\bar{u}_2 - \sum_{i=1}^n({u}_1)_i(\bar{u}_2)_i \bar\omega_{i}
&\displaystyle 1 - \sum_{i=1}^n(u_2)_i (\bar{u}_2)_i \bar\omega_{i}
\end{pmatrix},
\end{equation*}
where 
\[\bar\omega_i = \begin{cases}
1, & \lvert (\zeta(\alpha))_i\rvert > \frac{\lambda}{\tau},\\
0, & \lvert (\zeta(\alpha))_i \rvert< \frac{\lambda}{\tau}.
\end{cases}
\]
Thus, the Clarke {differential} of $\mathcal{L}$ is given by
\begin{equation*}
\partial^C \mathcal{L}(\alpha) = \left\{
\begin{pmatrix}\displaystyle
1 + \sum_{i=1}^n (u_1)_i (\bar{u}_1)_i \hat\omega_{i} 
& \displaystyle  \sum_{i=1}^n(\bar{u}_1)_i (u_2)_i \hat\omega_{i} \\
\displaystyle u_1^T\bar{u}_2 - \sum_{i=1}^n({u}_1)_i(\bar{u}_2)_i \hat\omega_{i}
&\displaystyle 1 -  \sum_{i=1}^n(u_2)_i (\bar{u}_2)_i\hat \omega_{i}
\end{pmatrix}
\left\vert~
\hat{\omega}_i
\left\{
\begin{array}{ll}
=1  & \text{if}~\lvert(\zeta(\alpha))_i\rvert > \frac{\lambda}{\tau}, \\
\in [0,1] & \text{if}~ \lvert(\zeta(\alpha))_i\rvert = \frac{\lambda}{\tau},
\\
=0 &  \text{if}~ \lvert(\zeta(\alpha))_i  \rvert <  \frac{\lambda}{\tau}.
\end{array}
\right.\right.
\right\}.
\end{equation*}
Therefore, we obtain  $V\in \partial^C \mathcal{L}(\alpha)$. 
\end{proof}

\end{appendices}

\end{document}